\documentclass[11pt]{amsart}
\usepackage{enumitem}

\numberwithin{equation}{section}
\newcommand{\IR}{\mathbb{R}}
\newcommand{\dotcup}{\mathbin{\mathaccent\cdot\cup}}
\newtheorem*{claim}{Claim}

\newtheorem{theorem}{Theorem}[section]
\newtheorem{definition}[theorem]{Definition}
\newtheorem{proposition}[theorem]{Proposition}
\newtheorem{lemma}[theorem]{Lemma}
\newtheorem{corollary}[theorem]{Corollary}

\def\eps{\varepsilon}
\def\al{\aligned}
\def\eal{\endaligned}
\def\be{\begin{equation}}
\def\ee{\end{equation}}
\def\lab{\label}
\def\a{\alpha}

\def\M{{\bf M}}

\def\eps{\varepsilon}
\def\al{\aligned}

\def\p{\partial}
\def\d{\nabla}

\DeclareMathOperator{\diam}{diam}
\DeclareMathOperator{\Ric}{Ric}
\DeclareMathOperator{\Rm}{Rm}
\DeclareMathOperator{\supp}{supp}
\DeclareMathOperator{\vol}{vol}
\DeclareMathOperator{\dist}{dist}
\DeclareMathOperator{\Hess}{Hess}
\newcommand{\rrm}{r_{|{\Rm}|}}

\numberwithin{equation}{section}

\begin{document}

\title[Heat kernel and curvature bounds in Ricci flows]{Heat kernel and curvature bounds in Ricci flows with bounded scalar curvature}
\author{Richard H. Bamler and Qi S. Zhang}
\address{Department of Mathematics, UC Berkeley, Berkeley, CA 94720,  USA}
\email{rbamler@math.berkeley.edu}
\address{Department of Mathematics, University of California, Riverside, CA 92521, USA}
\email{qizhang@math.ucr.edu}
\date{\today}

\begin{abstract}
In this paper we analyze Ricci flows on which the scalar curvature is globally or locally bounded from above by a uniform or time-dependent constant.
On such Ricci flows we establish a new time-derivative bound for solutions to the heat equation.
Based on this bound, we solve several open problems: 1. distance distortion estimates, 2. the existence of a cutoff function, 3. Gaussian bounds for heat kernels, and, 4. a backward pseudolocality theorem, which states that a curvature bound at a later time implies a curvature bound at a slightly earlier time.

Using the backward pseudolocality theorem, we next establish a uniform $L^2$ curvature bound in dimension $4$ and we show that the flow in dimension $4$ converges to an orbifold at a singularity.
We also obtain a stronger $\eps$-regularity theorem for Ricci flows.
This result is particularly useful in the study of K\"ahler Ricci flows on Fano manifolds, where it can be used to derive certain convergence results.
\end{abstract}
\maketitle
\tableofcontents

\section{Introduction}
In this paper we analyze Ricci flows on which the scalar curvature is locally or globally bounded by a time-dependent or time-independent constant.
We will derive new heat kernel and curvature estimates for such flows and point out the interplay between these two types of estimates.

We briefly summarize the main results of this paper.
More details can be found towards the end of the introduction.
Consider a Ricci flow  $(\M, (g_t)_{t \in [0,T)})$, $\partial_t g_t = -2 \Ric_{g_t}$ and assume that the scalar curvature $R$ is locally bounded from above by a given constant $R_0$.
We refer to the statements of the theorems for a precise characterization what we mean by ``locally''. 
Assuming the bound on the scalar curvature, we obtain the following results:
\begin{description}
\item[\it Distance distortion estimate]  The distance between two points $x, y \in \M$ changes by at most a uniform factor on a time-interval whose size depends on $R_0$ and the distance between $x, y$ at the central time (see Theorem \ref{thdistcon}).
This estimate addresses a question that was first raised by R. Hamilton.
\item[\it Construction of a cutoff function] We construct a cutoff function in space-time, whose support is contained in given a parabolic neighborhood and whose gradient, Laplacian and time-derivative are bounded by a universal constant (see Theorem \ref{Thm:cutoffnoLp}).
\item[\it Gaussian bounds on the heat kernel] Assuming a global bound on the scalar curvature, we deduce that the heat kernel $K(x,t;y,s)$ on $\M \times [0,T)$ is bounded from above and below by an expression of the form $a_1 (t-s)^{-n/2} \exp (- a_2 d_s^2 (x,y) /  (t-s))$ for certain positive constants $a_1, a_2$ (see Theorem \ref{thkernelUP}).
\item[\it Backward Pseudolocality Theorem] Named in homage to Perelman's forward pseudolocality theorem, this theorem states that whenever the norm of the Riemannian curvature tensor on an $r$-ball at time $t$ is less than $r^{-2}$, then the Riemannian curvature is less than $K r^{-2}$ on a smaller ball of size $\varepsilon r$ at earlier times $[ t - (\varepsilon r)^2, t]$ (see Theorem \ref{Thm:backwpseudoloc}).
\item[\it Strong $\varepsilon$-regularity theorem] If the local isoperimetric constant at a time-slice is close to the Euclidean constant, then we have a curvature bound at that time (see Corollary \ref{costrongreg}).
\item[\it $L^2$-curvature bound in dimension 4] If the scalar curvature is globally bounded on $\M \times [0, T)$, then the $L^2$-norm of the Riemannian curvature tensor, $\Vert {\Rm} (\cdot, t) \Vert_{L^2(M, g_t)}$, is bounded by a uniform constant, which is independent of time (see Theorem \ref{Thm:L2bound}).
\item[\it Convergence to an orbifold in dimension $4$] If the scalar curvature is globally bounded on $\M \times [0, T)$, then the metric $g_t$ converges to an orbifold with cone singularities as $t \nearrow T$ (see Corollary \ref{Cor:convergenceorbifold}).
\end{description}

Let us motivate the study of Ricci flows with bounded scalar curvature.
Consider a Ricci flow $(\M, (g_t)_{t \in [0,T)})$ that develops a singularity at time $T < \infty$.
It was shown by Hamilton (cf \cite{Ha:0}), that the maximum of the norm of the Riemannian curvature tensor diverges as $t \nearrow T$.
Later, Sesum (cf \cite{Sesum-Ricci}) showed that also the norm of the Ricci tensor has to become unbounded as $t \nearrow T$.
It remained an open questions whether the scalar curvature becomes unbounded as well.
In dimensions 2, 3 and in the K\"ahler case (cf \cite{ZhouZhang-scal}) this is indeed the case.
In order to understand the higher dimensional or non-K\"ahler case, it becomes natural to analyze Ricci flows with uniformly bounded scalar curvature and to try to rule out the formation of a singularity.

Another motivation for the scalar curvature bound arises in the study of K\"ahler-Ricci flows on Fano manifolds.
Let $\M$ be a complex manifold with $c_1 ( \M ) > 0$ and $g_0$ a K\"ahler-metric on $\M$ such that the corresponding K\"ahler form satisfies $\omega_0 \in c_1 ( \M)$.
Then $g_0$ can be evolved to a Ricci flow on the time-interval $[0,\frac12)$ and the scalar curvature satisfies the (time-dependent) bound $- C < R( \cdot, t) < C (\frac12 - t)^{-1}$ (cf \cite{ST:1}).
In other words, if we consider the volume normalized Ricci flow $(\widetilde{g}_t )_{t \in [0, \infty)}$, $\widetilde{g}_t = e^{t} g_{ (1- e^{-t}) / 2}$,
\[ \partial_t \widetilde{g}_t = - \Ric + \widetilde{g}_t, \]
then the scalar curvature of $\widetilde{g}_t$ satisfies $- C/t < R (\cdot, t) < C$ for some uniform constant $C$.

Ricci flows with bounded scalar curvature have been previously studied in \cite{CH:2}, \cite{CW:1}, \cite{CW:2}, \cite{CW:3}, \cite{Wa:1}, \cite{Z11:1}.
We will particularly use a result obtained by the second author in \cite{Z11:1}, which gives us upper volume bounds for small geodesic balls and certain upper an lower bounds on heat kernels for Ricci flows with bounded scalar curvature.

Before presenting our main results in detail, we introduce some notation that we will frequently use in this paper. 
We use $\M$ to denote a, mostly compact, Riemannian manifold and $g_t$ to denote the metric at time $t$.
For any two points $x, y \in \M$ we denote by $d_t (x,y)$ the distance between $x, y$ with respect to $g_t$ and for any $r \geq 0$ we let $B(x, t, r) = \{ y \in \M \; : \;  d_t (x,y) < r \}$ be the geodesic ball around $x$ of radius $r$ at time $t$.
For any $(x,t) \in \M \times [0,T)$, $r \geq 0$ and $\Delta t \in \IR$ we denote by
\[ P(x,t, r, \Delta t) = B(x,t,r) \times [t, t+ \Delta t] \quad \text{or} \quad B(x,t,r) \times [t + \Delta t, t], \]
depending on the sign of $\Delta t$, the \emph{parabolic neighborhood} around $(x,t)$.
For any measurable subset $S \subset \M$ and any time $t$, we denote by $|S|_t$ the measure of $S$ with respect to the metric $g_t$.
We use $\nabla$ and $\Delta$ to denote the gradient and the Laplace-Beltrami operators.  
Sometimes we will employ the notation $\nabla_t$, $\Delta_t$ or $\nabla_{g_t}$, $\Delta_{g_t}$ to emphasize the dependence on the time/metric.
We also reserve $R=R(x, t)$ for the scalar curvature of $g_t$, $\Ric = R_{ij}$ for the Ricci curvature and $\Rm = \Rm_{ijkl}$ for the Riemannian curvature tensor.
Lastly, we define $\nu [ g, \tau ] := \inf_{0 < \tau' < \tau} \mu [g, \tau']$, where $\mu [g,\tau]$ denotes Perelman's $\mu$-functional.
For more details on these functionals see section \ref{sec:Preliminaries}.

The first main result of this paper is a local bound on the distortion of the distance function between two points, given a local bound on the scalar curvature.
This result solves a basic question in Ricci flow, which was first raised by R. Hamilton (cf \cite[section 17]{Ha:1}), under a minimal curvature assumption.
We remark that previously, distance distortion bounds have been obtained under certain growth conditions on the curvature tensor or boundedness of parts of the Ricci curvature, see  \cite[section 17]{Ha:1},  \cite[Lemma 8.3]{P:1}, \cite{Si:1} and \cite{Tian-Wang}.
Observe that the following distance distortion bound becomes false if we drop the scalar curvature bound, as one can observe near a $3$-dimensional horn (for more details see the comment after the proof of Theorem \ref{thdistcon}).

\begin{theorem}[short-time distance distortion estimate] \lab{thdistcon} 
Let $(\M^n, (g_t)_{t \in [0,T)})$, $T < \infty$ be a Ricci flow on a compact $n$-manifold.
Then there is a constant $0 < \alpha <1$, which only depends on $\nu [g_0, 2T], n$, such that the following holds:

Suppose that $t_0 \in [0,T)$, $0 < r_0 \leq  \sqrt{t_0}$, let $x_0, y_0 \in \M$ be points with $d_{t_0} (x_0, y_0 ) \geq r_0$ and let $t \in [t_0 - \alpha r_0^2, \min \{ t_0 + \alpha r_0^2, T \})$.
Assume that $R \leq r_0^{-2}$ on $U \times [t, t_0]$ or $U \times [t_0, t]$, depending on whether $t \leq t_0$ or $t \geq t_0$, where $U \subset \M$ is a subset with the property that $U$ contains a time-$t'$ minimizing geodesic between $x_0$ and $y_0$ for all times $t'$ between $t_0$ and $t$.

Then
\[ \alpha d_{t_0} (x_0, y_0) < d_t (x_0, y_0) < \alpha^{-1} d_{t_0} (x_0, y_0). \]
\end{theorem}

Note that $U$ could for example be the ball $B(x_0, t_0, \alpha^{-1} d_{t_0} (x_0, y_0))$.
The upper bound of Theorem \ref{thdistcon} can be generalized for longer time-intervals.
For simplicity, we are phrasing this result using a \emph{global} scalar curvature bound, but localizations are possible.

\begin{corollary}[long-time distance distortion estimate] \lab{coBDP}
Let $(\M^n, (g_t)_{t \in [0,T)})$, $T< \infty$ be a Ricci flow on a compact $n$-manifold that satisfies $R\leq R_0 < \infty$ everywhere.
Then there is a constant $A < \infty$, which only depends on $\nu[g_0, 2T], n$, such that the following holds:

Suppose that $t_0 \in [0,T)$, let $x_0, y_0 \in \M$ be points and set $r_0 = d_{t_0} (x_0, y_0 )$.
For any $t \in [r_0^2, T)$ with $A^2 (r_0^2 + | t - t_0|)  \leq \min \{ R_0^{-1}, t_0 \}$ we have
\[ d_t (x_0, y_0 ) < A \sqrt{ r_0 ^2+ | t - t_0|}. \]
\end{corollary}

Similar techniques also imply the existence of a well behaved space-time cutoff function.

\begin{theorem}[cutoff function] \label{Thm:cutoffnoLp}
Let $(\M^n, (g_t)_{t \in [0,T)} )$, $T < \infty$ be a Ricci flow on a compact $n$-manifold.
Then there is a constant $\rho > 0$, which only depends on $\nu[g_0, 2T]$ and $n$, such that the following holds:

Let $(x_0, t_0) \in \M \times [0,T)$ and $0 < r_0 \leq \sqrt{t_0}$ and let $0 < \tau \leq  \rho^2 r_0^2$.
Assume that $R \leq r_0^{-2}$ on $P(x_0, t_0, r_0, - \tau)$.
Then there is a function $\phi \in C^\infty (\M \times [t_0 - \tau,t_0])$ with the following properties:
\begin{enumerate}[label=(\alph*)]
\item $0 \leq \phi < 1$ everywhere.
\item $\phi > \rho$ on $P(x_0, t_0, \rho r_0, - \tau)$.
\item $\phi = 0$ on $(\M \setminus B(x_0, t_0, r_0) ) \times [t_0 - \tau, t_0]$.
\item $|\nabla \phi | < r_0^{-1}$ and $|\partial_t \phi| + |\Delta \phi | < r_0^{-2}$ everywhere.
\end{enumerate}
\end{theorem}

The proofs of Theorem \ref{thdistcon}, Corollary \ref{coBDP} and Theorem \ref{Thm:cutoffnoLp} can be found in section \ref{sec:heatkernelcutoff}.

Next, we analyze the kernel $K(x,t; y,s)$ for the heat equation coupled with Ricci flow and establish Gaussian bounds on $K(x,t; y,s)$ and its gradient.
This addresses a question of Hein and Naber (cf \cite[Remark 1.15]{HN:1}).
Over the last few decades, similar questions have been subject to active research, especially after Li-Yau's paper \cite{LY:1}.
Among numerous useful papers, let us mention \cite{BCG:1}, \cite{ChH:1}, \cite{GH:2}, \cite{LT:1}, \cite{HN:1}, the books \cite{L:1} and \cite{Gr:2} and the reference therein.
Previous bounds usually relied on boundedness assumptions of the Ricci curvature and distance functions, see for instance \cite{Z06:1}.
Our proof makes use of a recent integral bound for the heat kernel obtained by Hein-Naber in \cite{HN:1} as well as our distance distortion bounds.
We mention that a lower bound on the heat kernel has been proven in \cite{Z11:1}, which matches the upper bound in this paper up to constants. 
Since many geometric quantities such as the scalar curvature obey equations of heat-type, we expect this result to have further applications.
For example, our bounds enable us to convert integral curvature bounds into pointwise bounds in certain settings.

\begin{theorem} \label{thkernelUP} 
Let $(\M^n, (g_t)_{t \in [0,T)})$, $T < \infty$ be a Ricci flow on a compact $n$-manifold that satisfies $R \leq R_0 < \infty$ everywhere.
Then for any $A < \infty$ there are constants $C_1 = C_1(A), C_2 = C_2(A) < \infty$, which only depend on $A, \nu[g_0, 2T], n$, such that the following holds:

Let $K(x, t; y, s)$ be the fundamental solution of the heat equation coupled with the Ricci flow (see section \ref{sec:Preliminaries} for more details), where $0 \leq s < t < T$.
Suppose that $t-s \leq A R_0^{-1}$ and $s > A^{-1} (t-s)$.
Then
\[
\al  
K(x, t; y, s) &> \frac{1}{C_1(t-s)^{n/2}} \exp \Big( {- \frac{C_2 d^2_s (x, y)}{t-s}} \Big), \\
K(x, t; y, s) &< \frac{C_1}{(t-s)^{n/2}} \exp \Big({ - \frac{d^2_s (x, y)}{C_2 (t-s)}} \Big),\\
 |\nabla_x K(x, t; y, s) |_{g_t} &<
\frac{C_1}{(t-s)^{(n+1)/2}} \exp \Big( {- \frac{d^2_s (x, y)}{C_2 (t-s)}} \Big). \eal
\] 
In the last line the gradient is taken with respect to the metric $g_t$.
\end{theorem}

Note that the lower bound for $s$ is necessary since we do not assume curvature or injectivity radius bounds on the initial metric.
On the other hand, if we assume such bounds on the initial metric, then one can derive upper and lower Gaussian bounds by standard methods, since $(M, g_t)$ will have bounded geometry at least for small times.
Using the reproducing formula and Theorem \ref{thkernelUP} we can derive Gaussian bounds up to any finite time.
Observe also that by Theorem \ref{thdistcon}, one can replace the distance $d_s (x, y)$ by $d_t (x, y)$ freely after adjusting the constant $C_2$ in the above statements.
This will be made clear during the proof of the theorem.
The proof of this theorem and further useful results, such as mean value inequalities for heat equations on Ricci flows, can be found in sections \ref{sec:meanvalue} and \ref{sec:heatkernelbound}.

Next, we prove the following backward pseudolocality theorem assuming a local scalar curvature bound.
Previously, similar backward pseudolocality properties have been established in two other settings:
First, Perelman (cf \cite[Proposition 6.4]{P:2}) obtained a similar result in the three dimensional case without assuming a scalar curvature bound. 
Second, X. X. Chen and B. Wang (cf \cite{CW:1}) proved a backward curvature bound under the additional assumption that the curvature tensor has uniformly bounded $L^{n/2}$ norm. 
Recently, the same authors presented a long-time backward pseudolocality property for K\"ahler Ricci flows on Fano manifolds (cf \cite{CW:3}).

\begin{theorem}[backward pseudolocality] \label{Thm:backwpseudoloc} \label{thbackcurv}
Let $(\M^n, (g_t)_{t \in [0,T)})$, $T < \infty$ be a Ricci flow on a compact manifold.
Then there are constants $\varepsilon > 0$, $K < \infty$, which only depend on $\nu [ g_0, 2T ], n$, such that the following holds:

Let $(x_0, t_0) \in \M \times (0, T)$ and $0 < r_0 \leq  \sqrt{t_0}$ and assume that
\begin{equation} \label{eq:pseudolocscalbound}
 R \leq r_0^{-2} \qquad \text{on} \qquad P(x_0, t_0, r_0, - 2 (\varepsilon r_0)^2) 
\end{equation}
and
\[ |{\Rm} (\cdot, t_0)| \leq r_0^{-2} \qquad \text{on} \qquad B(x_0,t_0,r_0). \]
Then
\[ |{\Rm}| < K r_0^{-2} \qquad \text{on} \qquad P(x_0,t_0,\varepsilon r_0, - (\varepsilon r_0)^2). \]
\end{theorem}
The proof of this theorem can be found in section \ref{sec:pseudoloc}.
Note that it can be observed from this proof that the factor $2$ in (\ref{eq:pseudolocscalbound}) can be replaced by any number larger than $1$ if the constants $\varepsilon$, $K$ are adjusted suitably.

As an application, the backward pseudolocality theorem can be coupled with Perelman's forward pseudolocality theorem (cf \cite{P:1}) to deduce a stronger $\eps$-regularity theorem for Ricci flows. 

\begin{corollary}[strong $\eps$-regularity] \label{costrongreg}
Let $(\M^n, (g_t)_{t \in [0,T)})$ be a Ricci flow on a compact $n$-manifold.
Then there are constants $\delta, \varepsilon > 0$ and $K < \infty$, where $\delta$ only depends on $n, \varepsilon$ and $K$ only depends on $\nu[g_0, 2T], n$, such that the following holds:

Let $(x_0, t_0) \in \M \times [0, T)$ and $0 < r_0 \leq \min \{  \sqrt{t_0}, \sqrt{T-t_0} \}$ and assume that
\[ R \leq r_0^{-2} \qquad \text{on} \qquad B(x_0, t_0, r_0) \times [t_0 - 2(\varepsilon r_0)^2, t_0 + (\varepsilon r_0)^2]. \]
and
\[ |\p \Omega|_{t_0}^n \ge (1-\delta) c_n |\Omega|_{t_0}^{n-1} \qquad \text{for any} \qquad\Omega \subset  B(x_0,t_0,r_0), \]
where $c_n$ is the Euclidean isoperimetric constant.
Then
\[ |{\Rm}| < K r_0^{-2} \qquad \text{on} \qquad
 B(x_0,t_0, \varepsilon r_0 ) \times [ t_0-  (\varepsilon r_0)^2, t_0 + (\varepsilon r_0)^2].  \]
\end{corollary}
The proof of this corollary can be found in section \ref{sec:pseudoloc}.

Note that in Perelman's forward pseudolocality theorem, the bound on the curvature tensor is $K/(s-t)$ for $s \in (t, t+ \varepsilon r^2]$, which blows up at time $t$. 
In contrast, the curvature bound in the corollary above is indepent of time and extends in both directions in time. 
As an application, Corollary \ref{costrongreg} combined with Shi's curvature derivative bound \cite{Sh:1}, seems to simplify and fill in some details in section 3.3 of the paper \cite{TZz:1} by G. Tian and Z.~L. Zhang. 
For example, the curvature bound (3.45) there now holds in a fixed open set relative
to the time level $0$ instead of the variable time $t_j$.
Note that if the $C^\a$ harmonic radius at a point $x$ and time $t$ is $r$, then the isoperimetric condition holds in $B(x, t, \theta r)$ for a fixed $\theta \in (0, 1)$.
Thus the corollary implies that the curvature is bounded in $B(x, t, \eps \theta r)$. 

As further application of the backward pseudolocality Theorem, we derive an $L^2$-bound for the Riemannian curvature in dimension 4, assuming a uniform bound on the scalar curvature.
The precise statement of the result makes use of the following notion:

\begin{definition} \label{Def:rho}
For a Riemannian manifold $(\M, g)$ and a point $x \in \M$ we define
\[ \rrm (x) := \sup \big\{ r > 0 \;\; : \;\; |{\Rm}| < r^{-2} \quad \text{on} \quad B(x, r) \big\}. \]
If $(\M, g)$ is flat, then we set $\rrm (x) = \infty$.
If $(\M, (g_t)_{t \in [0,T)})$ is a Ricci flow and $(x,t) \in \M \times [0,T)$, then $\rrm (x,t)$ is defined to be the radius $\rrm (x)$ on the Riemannian manifold $(\M , g_t)$.
\end{definition}
 
The result is now the following.

\begin{theorem}[$L^2$ curvature bound in dimension 4] \label{Thm:L2bound}
Let $(\M^4, (g_t)_{t \in [0,T)})$ be a Ricci flow on a compact $4$-manifold that satisfies $R  \leq R_0 < \infty$ everywhere.
Then there are constants $A, B < \infty$, which only depend on the product $R_0 T$ and on $\nu[g_0, 2T]$, such that the following holds:

Let $\chi (M)$ be the Euler characteristic of $\M$.
Then for all $t \in [T/2, T)$ the following bounds hold:
\[ \Vert {\Rm} (\cdot, t) \Vert_{L^2 (M, g_t)} = \bigg( \int_{\M} |{\Rm}(\cdot, t) |^2 dg_t \bigg)^{1/2}   \leq A \chi(M) + B \vol_0 M. \]
For any $p \in (0, 4)$
\[ \int_{\M} | {\Ric} |^p(x, t) dg_t + \int_{\M} \rrm^{-p} (x, t)  d g_t \leq A \chi ( M  )+ \frac{B}{4-p} \vol_0 M. \]
Finally, for all $0 < s \leq 1$
\[ \frac{| \{ |{\Ric}| (\cdot, t) \geq s^{-1} \} |_t }{s^4} + \frac{| \{ \rrm (\cdot, t) \leq s \} |_t }{s^4}  \leq A \chi ( M ) + B \vol_0 M . \]
Here we use the short form $\{ f \geq a \} := \{ x \in \M \; : \; f(x) \geq a \}$.
\end{theorem}
The proof of this theorem can be found in section \ref{sec:L2bound}.
Note that the first bound of Theorem \ref{Thm:L2bound} has been obtained independently in \cite{Simon:1}.

Since the $L^2$-bound of $\Rm$ is scaling invariant and descends to geometric limits, we immediately obtain that singularity models of $4$-dimensional Ricci flows with bounded scalar curvature have $L^2$-bounded curvature as well:

\begin{corollary} \label{Cor:blowupALE}
Let $(\M^4, (g_t)_{t \in [0,T)})$, $T < \infty$ be a Ricci flow on a compact $4$-manifold that satisfies $R  \leq R_0 < \infty$ everywhere.
Let $(x_k, t_k) \in \M \times [0,T)$ be a sequence with $Q_k = |{\Rm}| (x_k, t_k) \to \infty$ and assume that the pointed sequence of blow-ups $(\M, Q_k g_{(Q_k^{-1} t + t_k)}, x_k)$ converges to some ancient Ricci flow $(\M_\infty, (g_{\infty,t})_{t \in (-\infty, 0]}, x_\infty)$ in the smooth Cheeger-Gromov sense.
Then $g_{\infty, t} = g_\infty$ is constant in time and $(\M_\infty, g_\infty )$ is Ricci-flat and asymptotically locally Euclidean (ALE).
\end{corollary}

Corollary \ref{Cor:blowupALE} also follows from recent work of Cheeger and Naber (cf \cite{Cheeger-Naber-codim-4}).
A direct consequence of Corollary \ref{Cor:blowupALE} is (see \cite[Corollary 5.8]{Anderson}):

\begin{corollary}
Let $(\M^4, (g_t)_{t \in [0,T)})$, $T < \infty$ be a Ricci flow on a compact $4$-manifold $\M$ that satisfies the following topological condition: the second homo\-logy group over every field vanishes, i.e. $H_2 (\M; \mathbb{F}) = 0$ for every field $\mathbb{F}$ (for example, the $4$-sphere satisfies this condition).

Then the scalar curvature becomes unbounded as $t \nearrow T$.
\end{corollary}

The $L^2$-bound on the Riemannian curvature can also be used to understand the formation of the ALE-singularities on a global scale:

\begin{corollary} \label{Cor:convergenceorbifold}
Let $(\M^4, (g_t)_{t \in [0,T)})$, $T < \infty$ be a Ricci flow on a compact $4$-manifold that satisfies $R \leq R_0 < \infty$ everywhere.
Then $(\M, g_t)$ converges to an orbifold in the smooth Cheeger-Gromov sense.

More specifically, we can find a decomposition $\M = \M^{\textnormal{reg}} \dotcup \M^{\textnormal{sing}}$ with the following properties:
\begin{enumerate}[label=(\alph*)]
\item $\M^{\textnormal{reg}}$ is open and connected.
\item $\M^{\textnormal{sing}}$ is a null set with respect to $g_t$ for all $t \in [0,T)$.
\item $g_t$ smoothly converges to a Riemannian metric $g_T$ on $\M^{\textnormal{reg}}$.
\item $(\M^{\textnormal{reg}}, g_T)$ can be compactified to a metric space $(\overline{\M}^{\textnormal{reg}}, \overline{d})$ by adding finitely many points and the differentiable structure on $\M^{\textnormal{reg}}$ can be extended to a smooth orbifold structure on $\overline{\M}^{\textnormal{reg}}$, such that the orbifold singularities are of cone type.
\item Around every orbifold singularity of $(\overline{\M}^{\textnormal{reg}}, \overline{d})$ the metric $g_T$ satisfies $|\nabla^m {\Rm} | < o (r^{-2-m})$ and $|\nabla^m {\Ric}| < O(r^{-1-m})$ as $r \to 0$, where $r$ denotes the distance to the singularity.
Furthermore, for every $\eps > 0$ we can find a smooth orbifold metric $\overline{g}_\eps$ on $\overline{\M}^{\textnormal{reg}}$ (meaning that $\overline{g}_\eps$ pulls back to a smooth Riemannian metric on local orbifold covers) such that the following holds:
\[ \Vert g_T - \overline{g}_\eps \Vert_{C^0(\M^{\textnormal{reg}}; \overline{g}_\eps )} + \Vert g_T - \overline{g}_\eps \Vert_{W^{2,2}(\M^{\textnormal{reg}}; \overline{g}_\eps)} < \eps \]
Here, the $C^0$ and $W^{2,2}$-norms are taken with respect to $\overline{g}_\eps$.
\end{enumerate}
\end{corollary}
A slightly weaker characterization was obtained independently in \cite{Simon:2} and in \cite{CW:1} under an additional curvature assumption.
The proof of this corollary can be found in section \ref{sec:L2bound}.

This paper is organized as follows:
In section \ref{sec:Preliminaries}, we introduce the notions and tools that we will use throughout the paper.
In section \ref{sec:heatkernelcutoff}, we prove a time derivative bound for solutions of the heat equation on a Ricci flow, which is independent of the Ricci curvature.
Then we use this bound to derive the distance distortion estimates, Theorem \ref{thdistcon} and Corollary \ref{coBDP}, and construct a well behaved cutoff function in space-time in Theorem \ref{Thm:cutoffnoLp}.
Next, in section \ref{sec:heatkernelbound}, we use this cutoff function to establish mean value inequalities for solutions of the heat and conjugate heat equation in section \ref{sec:meanvalue} and eventually the Gaussian bounds for the heat kernel, Theorem \ref{thkernelUP}.
In section \ref{sec:pseudoloc}, we prove the backward pseudolocality Theorem \ref{thbackcurv} and the strong $\varepsilon$-regularity theorem, Corollary \ref{costrongreg}.
Finally, we prove the $L^2$ curvature bound in dimension 4, Theorem \ref{Thm:L2bound}, and its consequences.

\section{Preliminaries} \label{sec:Preliminaries}
Consider a Ricci flow $(\M, (g_t)_{t \in [0,T)})$, $T < \infty$ on a compact $n$-manifold $\M$.
For any $(x_0, t_0) \in \M \times [0,T)$, $r > 0$, we will frequently use the notation
\begin{alignat*}{3}
 Q^+(x_0, t_0, r) &= \{ (x, t) \in \M \times [0,T) \;\; : \;\; d_t (x_0, x)<r, \quad & t_0 \leq & t \leq t_0+r^2 \} \\
 Q^-(x_0, t_0, r) &= \{ (x, t) \in \M \times [0,T) \;\; : \;\; d_t (x_0, x)<r, \quad & t_0 - r^2 \leq & t \leq t_0 \} 
\end{alignat*}
to denote the ``forward'' and ``backward'' parabolic cubes.

Next, note that by applying the maximum principle to the evolution equation $\partial_t R = \Delta R + 2 |{\Ric}|^2 \geq \Delta R + \frac2{n} R^2$, we can deduce the following lower bound on the scalar curvature:
\[ R(\cdot, t) > - \frac{n}{2t}. \]
So, for instance, if we also have an upper bound of the form $R \leq R_0 < \infty$ and if we are working at a scale $0 < r_0 \leq \min \{ R_0^{-1/2}, \sqrt{t_0} \}$, then  we have the two-sided bound
\[ |R| \leq n r_0^{-2} \qquad \text{on} \qquad \M \times [t_0 - \tfrac12 r_0^2, t_0]. \]
In such a situation, we will often rescale parabolically and assume that $r_0 = 1$, $t_0 \geq 1$ and
\[ |R| \leq n \qquad \text{on} \qquad \M \times [t_0 - \tfrac12, t_0]. \]

Assume for simplicity, that $|R| \leq R_0$ everywhere.
By the evolution equation of the volume form, $\partial_t dg_t = - R dg_t$, the bound on the scalar curvature implies the following distortion estimate for the volume element:
\begin{equation} \label{eq:Prel-volelementdistortion}
 e^{- R_0 |t-s|} dg_s \leq dg_t \leq e^{R_0 |t-s|} dg_s \qquad s, t \in [0,T).
\end{equation}
This implies that for any measurable subset $S \subset \M$, we have
\begin{equation} \label{eq:Prel-voldistortion}
 e^{- R_0 |t-s|} |S|_s \leq |S|_t \leq e^{R_0 |t-s|} |S|_s \qquad s, t \in [0,T).
\end{equation}

Next, we recall Perelman's $\mathcal{W}$-functional (cf \cite{P:1}):
\[ \mathcal{W} [ g, f, \tau ] = \int_{\M} \big( \tau ( |\nabla f|^2 + R) + f - n \big) (4 \pi \tau)^{-\frac{n}2} e^{-f} dg \]
Here $g$ is a Riemannian metric, $f \in C^1(\M)$ and $\tau > 0$.
The fact that this functional involves no curvature term other than the scalar curvature, comes very handy in the study of bounded scalar curvature.
Recall that the derived functionals
\[ \mu [g, \tau] = \inf_{\int_{\M} (4 \pi \tau)^{-\frac{n}2} e^{-f} dg = 1} \mathcal{W} [ g, f, \tau ] \]
and
\[ \nu [g, \tau] = \inf_{0 < \tau' < \tau} \mu [g, \tau'], \qquad \nu[g] = \inf_{\tau > 0} \mu [g, \tau] \]
are monotone in time and that $\nu [g, \tau], \nu[g] \leq 0$.
By this, we mean that $\partial_t \mu [g_t, \tau - t], \partial_t \nu [g_t, \tau - t], \partial_t \nu[g_t] \geq 0$.

We now mention two important consequences, which have been derived from this monotonicity and which we will be using frequently in this paper without further mention.
The first consequence is Perelman's celebrated No Local Collapsing Theorem (see \cite{P:1}), which can be phrased as follows:
Let $(x_0,t_0) \in \M \times [0,T)$, $0 < r_0 < \sqrt{t_0}$ and assume that $R < r_0^{-2}$ on $B(x_0, t_0, r_0)$.
Then
\begin{equation} \label{eq:kappa1}
  |B(x_0,t_0,r_0)|_{t_0} > \kappa_1 r_0^n,
\end{equation}
for some constant $\kappa_1$, which only depends on $\nu[g_0, 2T]$ and $n$.
On the other hand, a non-inflating property was shown in \cite{Z11:1} (see also \cite{CW:2}).
This property states that whenever $(x_0,t_0) \in \M \times [0,T)$, $0 < r_0 < \sqrt{t_0}$
\[ R (\cdot, t) \leq \frac{\alpha}{t_0 - t} \qquad \text{on} \qquad Q^- (x_0, t_0, r_0), \]
then
\begin{equation} \label{eq:kappa2}
 | B(x_0,t_0,r_0) |_{t_0} < \kappa_2 r_0^n,
\end{equation}
where $\kappa_2$ again only depends on $\nu[g_0, 2T]$ and $n$.

Next, we discuss heat equations coupled with the Ricci flow.
The forward heat equation
\begin{equation} \label{he}
\begin{cases}
\partial_t u - \Delta u = 0,\\
\p_t g_t = - 2 \Ric_{g_t}
\end{cases}
\end{equation}
has a conjugate,
\begin{equation} \label{che}
\begin{cases}
- \partial_t u - \Delta u + R u =0 \\
\p_t g_t = - 2 \Ric_{g_t},
\end{cases}
\end{equation}
which evolves backwards in time.
So for any two compactly supported functions $u, v \in C^\infty (\M \times (0,T))$ we have
\[ \int_{\M \times (0,T)} (\partial_t u - \Delta u) \cdot v dg_t dt = \int_{\M \times (0,T)} u \cdot (- \partial_t v - \Delta v + Rv ) dg_t dt. \]
We will denote by $K (x, t; y, s)$, where $x, y \in \M$, $0\leq s < t < T$, the heat kernel associated with the heat equation (\ref{he}), meaning that for any fixed $(y,s) \in \M \times [0,T)$
\begin{equation} \label{eq:hkdefinitionsec2}
 (\partial_t - \Delta_x ) K(\cdot, \cdot ; y, s) = 0 \qquad \text{and} \qquad \lim_{t \searrow s} K(\cdot, t; y,s) = \delta_y.
\end{equation}
Then $K(x,t; \cdot, \cdot)$ is the kernel associated to the conjugate heat equation (\ref{che}), meaning that for any fixed $(x,t) \in \M \times [0,T)$
\[ (-\partial_s - \Delta_y + R ) K(x, t ; \cdot, \cdot) = 0 \qquad \text{and} \qquad \lim_{s \nearrow t} K(x,t ; \cdot, s) = \delta_x. \]

In \cite[Theorem 3.2]{Z06:1} (see also \cite[Theorem 5.1]{CH:1}), the second author has obtained the following derivative bound for positive solutions $u \in C^\infty (\M \times (0,T))$ of the heat equation (\ref{he}):
\begin{equation} \label{eq:Prel-derbound}
 \frac{|\nabla u(x,t)|}{u(x,t)}  \leq \sqrt{\frac1t} \sqrt{\log \frac{J}{u(x,t)}},
\end{equation}
whenever $0 < u < J$ on $\M \times (0, t]$.
Note that this bound is sharp for the heat kernel on Euclidean space with $J = \sup_{\M \times (0,t]} u$ and does not assume boundedness of the scalar curvature.

Lastly, we mention two heat kernel estimates, which have been obtained by the second author in \cite[equations (1.5), (1.7), etc.]{Z11:1}.
The first estimate is a global upper bound on the heat kernel.
Assume that $x, y \in \M$, $0 \leq s < t < T$ and that we have the lower scalar curvature bound $R \geq - R_0$ on $\M \times [s,t]$.
Then
\begin{equation} \label{eq:Prel-upperhkbound}
 K (x,t; y,s) < \frac{C}{(t-s)^{n/2}},
\end{equation}
where $C$ only depends on $\nu[g_0, 2T], n$ and $(t-s) R_0$.

The second estimate is a distance dependent lower bound on the heat kernel.
This bound is a consequence of (\ref{eq:Prel-derbound}) and Perelman's Harnack inequality.
Let $x, y \in \M$, $0 \leq s < t < T$ and consider a smooth curve $\gamma : [s,t] \to \M$ between $(y,s)$ and $(x,t)$, i.e. $\gamma (s) = y$ and $\gamma(t) = x$.
Its $\mathcal{L}$-length is defined as
\[ \mathcal{L} (\gamma) := \int_{s}^t \sqrt{t-t'} \big( |\gamma'(t')|^2_{t'} + R(\gamma(t'), t') \big) dt'. \]
The reduced distance between $(x, t)$ and $(y,s)$ is defined as
\[ \ell_{(x,t)} (y,s) := \frac1{2\sqrt{t-s}} \inf \big\{ \mathcal{L} (\gamma) \;\; : \;\; \text{$\gamma : [s,t] \to \M$ between $(y,s)$ and $(x,t)$} \big\}. \]
Then by \cite[Corollary 9.5]{P:1} we have
\begin{equation} \label{eq:Prel-Perelmanslowerhkbound}
 K(x,t; y,s) \geq \frac{1}{(4\pi (t-s))^{n/2}} e^{- \ell_{(x,t)} (y,s)}.
\end{equation}
Applying (\ref{eq:Prel-Perelmanslowerhkbound}) for $x=y$ and (\ref{eq:Prel-derbound}) at time $t$ yields the following Gaussian lower bound on the heat kernel (see \cite{Z11:1} for more details):
Assume that we have the upper scalar curvature bound $R(y, t') \leq R_0$ for all $t' \in [s,t]$.
Then
\begin{equation} \label{eq:Prel-lowerhkbound}
 K(x,t; y,s) > \frac{C^{-1}}{(t-s)^{n/2}} \exp \bigg( {- \frac{ d_t^2(x,y)}{2(t-s)}} \bigg),
\end{equation}
where $C$ only depends on $(t-s) R_0$.

\section{Heat kernel bounds, distance distortion estimates and the construction of a cutoff function} \label{sec:heatkernelcutoff}
In this section, we first derive a bound on the time derivative (or Laplacian) of positive solutions of the heat equation, which is independent of the Ricci curvature.
This bound is then used to obtain distance distortion estimates and to construct a space-time cutoff function.

\begin{lemma} \lab{leddu}
Let $(\M^n, (g_t)_{t \in [0,T)})$ be a Ricci flow on a compact $n$-manifold and let $u \in C^\infty (\M \times [0,T))$ be a positive solution to the heat equation $\p_t u = \Delta u$, $u(\cdot, 0) = u_0$ that is coupled to the Ricci flow.
Then the following is true:
\begin{enumerate}[label=(\alph*)]
\item There is a constant $B < \infty$, depending only on the dimension of $\M$, such that the following holds:
If $0<u \le a$ on $\M \times [0, T]$ for some constant $a > 0$, then for all $(x, t) \in \M \times (0,T)$
\be
\lab{ddu<A/t}
\bigg( |\Delta u | + \frac{| \d u |^2}{u} - a R \bigg) (x, t) \le \frac{ Ba}{t}
\ee
\item If $0<u \le 1$ on $\M \times [0, T]$, then for all $(x,t) \in \M \times [0,T)$
\[ \bigg( \Delta u + \frac{| \d u |^2}{u} - R \bigg) (x, t) \le \sup_{\M} \bigg( \Delta u_0 + \frac{| \d u_0 |^2}{u_0} - R(\cdot, 0) \bigg)
+t \big( 1+ \tfrac{4}{n} \big)  \sup_{\M} \frac{| \d u_0 |^4}{u^2_0} \]
and
\[ \bigg( -\Delta u + \frac{| \d u |^2}{u} - R \bigg) (x, t) \le
\sup_{\M} \bigg( -\Delta u_0 + \frac{| \d u_0 |^2}{u_0} - R(\cdot, 0) \bigg)
+ t \sup_{\M} \frac{| \d u_0 |^4}{u^2_0}. \]
\end{enumerate}
\end{lemma}

\begin{proof} \textit{Part (a).}
Without loss of generality we may assume $a=1$ so that $0 < u \le 1$ everywhere.
For convenience of presentation we write
\[ L_1 = - \Delta u + \frac{|\d u|^2}{u} - R, \qquad L_2 = \Delta u + \frac{|\d u|^2}{u} - R. \]

In the following we will show that $L_1, L_2 \leq \frac{B}t$.
Using standard identities, it is easy to check that in a local orthonormal frame we have
 \begin{alignat}{1} 
 (\partial_t - \Delta ) \Delta u &= \Delta (\partial_t - \Delta ) u + (\p_t \Delta - \Delta \partial_t) u = 2 R_{ij} u_{ij} , \notag \\
 (\partial_t - \Delta ) R &= 2 | {\Ric} |^2, \notag \\
 (\partial_t - \Delta ) \frac{|\d u|^2}{u} &= - \frac{2}{u} \left| u_{ij} -\frac{ u_i u_j}{u} \right|^2. \label{eq:nabu2divu}
\end{alignat}
Therefore, we can write
\[
(\partial_t - \Delta ) L_1 = - \frac{2}{u} \left| u_{ij} -
  \frac{ u_i u_j}{u} \right|^2 - 2 R_{ij} u_{ij}  - 2 R_{ij}^2.
\] Rearranging the terms on the right-hand side and using the assumption that $u \le 1$,
we deduce
\begin{alignat*}{1}
(\partial_t - \Delta )
L_1 &\leq -\left| u_{ij} - \frac{ u_i u_j}{u} \right|^2 -
2 R_{ij} \left( u_{ij}   - \frac{ u_i u_j}{u} \right) -  R_{ij}^2
- 2 R_{ij} \frac{ u_i u_j}{u} -  R_{ij}^2\\
&= - \left| u_{ij} - \frac{ u_i u_j}{u} + R_{ij} \right|^2 - \left| \frac{ u_i u_j}{u} + R_{ij} \right|^2
+ \frac{|\d u|^4}{u^2}\\
&\leq - \frac{1}{n} \left( \Delta u - \frac{|\d u|^2}{u} + R \right)^2 + \frac{|\d u|^4}{u^2}.
\end{alignat*}
Thus
\be
\lab{HQ1>}
(\partial_t - \Delta ) L_1 \leq - \frac{1}{n} L^2_1 + \frac{|\d u|^4}{u^2}.
\ee

From (\ref{eq:Prel-derbound}), one has
\begin{equation} \label{eq:derivativeboundfromZ}
 \frac{|\d u|^2}{u}  \le \frac{u}{t} \log \frac{1}{u} \le \frac{1}{e t}.
\end{equation}
Using this and (\ref{HQ1>}), we deduce
\[ (\partial_t - \Delta ) L_1 \leq - \frac{1}{n} L^2_1 + \frac{1}{e^2 t^2}. \]
Let $B$ be a positive number such that
\[ \frac{B+e^{-2}}{B^2} = \frac{1}{n}. \]
Then
\[ (\partial_t - \Delta ) (B/t) = - \frac{B+e^{-2}}{B^2} (B/t)^2 + \frac{1}{e^2 t^2}=
- \frac{1}{n} (B/t)^2 + \frac{1}{e^2 t^2}. \]
Hence
\[ (\partial_t - \Delta ) \Big( L_1 - \frac{B}{t} \Big) \leq - \frac{1}{n} \Big(L_1+ \frac{B}{t} \Big) \Big(L_1- \frac{B}{t} \Big). \]
The maximum principle then implies $L_1 \le \frac{B}{t}$.

In order to prove the upper bound on $L_2$, we compute
\begin{alignat*}{1}
(\partial_t - \Delta ) L_2 &= - \frac{2}{u} \left| u_{ij} - \frac{ u_i u_j}{u} \right|^2 + 2 R_{ij} u_{ij}  - 2 R_{ij}^2\\
  &\leq - \left| u_{ij} - \frac{ u_i u_j}{u} \right|^2 +
2 R_{ij} \left( u_{ij}   - \frac{ u_i u_j}{u} \right) -  R_{ij}^2
+ 2 R_{ij} \frac{ u_i u_j}{u} - R_{ij}^2\\
&= - \left| u_{ij} - \frac{ u_i u_j}{u} - R_{ij} \right|^2 - \left| \frac{ u_i u_j}{u} - R_{ij} \right|^2
+ \frac{|\d u|^4}{u^2} \displaybreak[1] \\
&\leq - \frac{1}{n} \left( \Delta u - \frac{|\d u|^2}{u} - R \right)^2 +\frac{|\d u|^4}{u^2} \displaybreak[1] \\
&=- \frac{1}{n} \left( L_2  - 2 \frac{|\d u|^2}{u} \right)^2 + \frac{|\d u|^4}{u^2} \displaybreak[1] \\
&= - \frac1{2n} \left( L_2 - 4 \frac{|\nabla u|^2}{u} \right)^2 - \frac1{2n} L_2^2 + ( 1 + \tfrac4n ) \frac{|\nabla u |^4}{u^2} \displaybreak[1] \\
&\leq - \frac1{2n} L_2^2  + ( 1 + \tfrac4n ) \frac{|\nabla u |^4}{u^2}
\end{alignat*}
Therefore, using again (\ref{eq:derivativeboundfromZ}), we obtain
\[ (\partial_t - \Delta ) L_2 \leq - \frac{1}{2n} L^2_2  + \frac{1+ \tfrac{4}{n}}{e^2} \cdot \frac1{t^2}.
\]
We now argue similarly as before.
Choose $B > 0$ such that
\[ B^{-1} + \frac{1 + \frac4n}{e^2} B^{-2} = \frac1{2n}. \]
Then
\[ (\partial_t - \Delta ) (B/t) = - \frac1{2n} (B/t)^2 + \frac{1+ \tfrac{4}{n}}{e^2} \cdot \frac1{t^2}. \]
And hence
\[ (\partial_t - \Delta ) \Big( L_2 - \frac{B}t \Big) \leq - \frac1{2n} \Big( L_2 + \frac{B}{t} \Big) \Big( L_2 - \frac{B}{t} \Big). \]
So by the maximum principle $L_2 \leq \frac{B}{t}$, which proves (\ref{ddu<A/t}).

\textit{Part (b).}
By (\ref{eq:nabu2divu}), we know that the function $|\nabla u|^2/u$ is a subsolution to the heat equation.
So the maximum principle implies that
\[ \frac{|\nabla u|^2}{u} \le \sup_{\M} \frac{|\nabla u_0|^2}{u_0}. \]
The remaining estimates then follow using the bounds for $(\partial_t - \Delta )L_1$ and $(\partial_t - \Delta ) L_2$ from before along with the maximum principle.
\end{proof}

The previous bound on the gradient and Laplacian of solutions to the heat equation, can be used to derive a local bound on the distortion of the distance between two points.

\begin{proof}[Proof of Theorem \ref{thdistcon}]
\textit{Step 1.}
We first establish the upper bound on $d_t (x_0, y_0)$ in the case in which $r_0 \leq d_{t_0}(x_0, y_0) \leq 2 r_0$.
After rescaling the flow parabolically, we may assume without loss of generality that $r_0 = 1$.
Then we have $R \leq 1$ on $U \times [t, t_0]$ or $U \times [t_0, t]$ and $t_0 \geq 1$.
Moreover, as explained in section \ref{sec:Preliminaries}, we have the lower bound $R \geq -n$ on $\M \times [t_0 - \tfrac12, t_0]$.
Let $\gamma  : [0,1] \to U$ be a time-$t_0$ minimizing geodesic between $x_0, y_0$.

By \cite[end of 7.1]{P:1}, we can find a point $z \in \M$ such that the reduced distance between $(x_0, t_0)$ and $(z, t_0 - \tfrac12)$ satisfies
\[ l_{(x_0, t_0)} (z, t_0 - \tfrac12 ) \leq  \frac{n}2. \]
Consider the heat kernel $K(x,t) = K(x,t; z, t_0 - \frac12)$ centered at $(z, t_0- \frac12)$, i.e. $\partial_t K = \Delta K$ (see (\ref{eq:hkdefinitionsec2}) for more details).
First, we note that
\[ \frac{d}{dt} \int_{\M} K(\cdot, t) dg_t = \int_{\M} \big( \Delta K (\cdot, t) - R(\cdot, t) K (\cdot, t) \big) dg_t \leq n \int_{\M} K(\cdot, t) dg_t. \]
So
\begin{equation} \label{eq:L1boundk}
 \int_{\M} K(\cdot, t) dg_t \leq e^{n (t - (t_0 - \frac12))}.
\end{equation}
Recall also that by (\ref{eq:Prel-upperhkbound}) and (\ref{eq:Prel-Perelmanslowerhkbound}) there is a uniform constant $B_0 < \infty$ such that for all $t \in [t_0 - \frac14, t_0 + \frac14]$
\begin{multline} \label{eq:lowerbound}
 K(\cdot, t) < \frac{B_0}{(t-(t_0 - \frac12))^{n/2}} \quad \text{on} \quad \M \qquad \text{and} \\ \qquad
K(x_0, t_0) \geq \frac{1}{(4\pi (t_0 - (t_0 - \tfrac12))^{n/2}} e^{- l_{(x_0, t_0)} (z, t_0 - 1/2)} >  \frac{B_0^{-1}}{(2\pi)^{n/2}}.
\end{multline}
By Lemma \ref{leddu}(a) applied to the region $U \times [t, t_0]$ or $U \times [t_0, t]$, we get that there is a uniform constant $B < \infty$ such that
\begin{equation} \label{eq:nabkdtk2bound}
 K < B \quad \text{on} \quad \M \times [t, t_0] \qquad \text{and} \qquad |\partial_t K| < B \quad \text{on} \quad U \times [t, t_0 ].
\end{equation}

We also recall the inequality (\ref{eq:Prel-derbound})
\[ \frac{|\nabla K (\cdot, t')|}{K (\cdot, t')} < \sqrt{\frac1{t' - (t_0 -\frac18)}} \sqrt{\log \frac{B}{K(\cdot, t')}}, \]
which holds on all of $\M$ for $t' \in [t_0 - \frac1{8}, t_0 + \frac1{8}]$.
We can rewrite this inequality as
\begin{equation} \label{eq:derivativeboundhk}
 \Bigg| \nabla \sqrt{ \log  \frac{B}{K(\cdot, t')}  } \Bigg| < \frac{1}{2\sqrt{t' - (t_0 - \frac18)}}.
\end{equation}
Integrating this bound at time $t_0$ along $\gamma$ and using the lower bound in (\ref{eq:lowerbound}) yields that
\begin{equation} \label{eq:k>c}
 K(\cdot, t_0) > c \qquad \text{on} \qquad \gamma([0,1])
\end{equation}
for some uniform constant $c > 0$.
Assume now that $\alpha \leq \min \{ \frac18, \frac12 B^{-1} c \}$ and fix some $t \in [t_0 - \alpha, t_0 + \alpha]$.
Then by (\ref{eq:nabkdtk2bound}) and (\ref{eq:k>c})
\[ K(\cdot, t) > c/2 \qquad \text{on} \qquad \gamma([0,1]). \]
Due to the derivative bound (\ref{eq:derivativeboundhk}) on $K (\cdot, t)$, we can choose a uniform $\beta > 0$ such that for all $s \in [0,1]$
\[ K(\cdot, t) >  c/4 \qquad \text{on} \qquad B(\gamma(s), t, \beta). \]

Let now $N \geq 1$ be maximal with the property that there are parameters $0 \leq s_1\leq \ldots \leq s_N \leq 1$ such that the balls $B(\gamma(s_1), t, \beta), \ldots, B(\gamma(s_N), t, \beta)$ are pairwise disjoint.
Then the balls $B(\gamma(s_1), t, 2\beta), \ldots, B(\gamma(s_N), t, 2 \beta)$ cover $\gamma([0,1])$.
We argue that this implies that $d_t (x_0, y_0) \leq N \cdot 4 \beta$:
Consider a graph consisting of vertices labeled by $1, \ldots, N$.
Connect vertex $i$ with vertex $j$ by an edge if $B(\gamma(s_i), t, 2 \beta)$ and $B(\gamma (s_j), t, 2 \beta)$ intersect.
Since $\gamma$ is covered by the balls $B(\gamma(s_i), t, 2 \beta)$, this graph has to be connected.
So it is possible to connect any two vertices by a chain of length $\leq N - 1$.
It follows that $d_t (\gamma(s_i), \gamma(s_j)) < (N-1) \cdot 4 \beta$ for all $i, j \in \{ 1, \ldots, N \}$.
Since $x_0, y_0$ are contained in the union of the $B(\gamma(s_i), t, 2 \beta)$, we obtain $d_t (x_0, y_0) < (N-1) \cdot 2 \beta + 2 \cdot 2 \beta = N \cdot 4 \beta$.

It remains to estimate $N$ from above.
The fact that the $B(\gamma(s_i), t, \beta)$ are pairwise disjoint implies, using (\ref{eq:L1boundk}) and (\ref{eq:kappa1}), that
\[ e^n >  e^{n(t - (t_0 - \frac12))} \geq \int_{\M} K(\cdot, t) dg_t \geq \sum_{i=1}^N \int_{B(\gamma(s_i), t, \beta)} K(\cdot, t) dg_t > N \cdot \kappa_1 \beta^n \cdot c/4 . \]
So it follows that
\[ N < \frac{e^n}{\kappa_1 \beta^{n} \cdot c/4} =: N_0. \]
Putting everything together yields (recall that by assumption $1 \leq d_{t_0} (x_0, y_0 ) \leq 2$).
\[ d_t (x_0, y_0) < 4 N_0 \beta < 8 N_0 \cdot \beta d_{t_0} (x_0, y_0). \]
This proves the upper bound for $\alpha \leq \min \{ \frac18,  \frac12 B^{-1} c, (4N_0)^{-1} \}$ in the case in which $r_0 \leq d_{t_0} (x_0, y_0) \leq 2 r_0$.

\textit{Step 2.}
The upper bound in the general case, $d_{t_0} (x_0, y_0) \geq r_0$, follows now easily:
Let $\gamma : [0,1] \to U$ be a time-$t_0$ minimizing geodesic and choose parameters $0 = s_0 < s_2 < \ldots < s_k = 1$ such that $r_0 \leq d_{t_0} (\gamma(s_{i-1}), \gamma(s_i)) < 2 r_0$.
Then for any $t \in [t_0 - \alpha r_0^2, \min \{ t_0 + \alpha r_0^2, T \})$ we have
\[ d_t (x_0, y_0) \leq \sum_{i=1}^k d_t (\gamma(s_{i-1}), \gamma (s_i)) < \sum_{i=1}^k \alpha^{-1} d_{t_0} (\gamma(s_{i-1}), \gamma (s_i)) = \alpha^{-1} d_{t_0} (x_0, y_0). \]

\textit{Step 3.}
Next, we derive the lower bound on $d_t (x_0, y_0)$ assuming that the upper bound holds for the constant $\alpha$ that we have chosen before.
We claim that for any $t \in [t_0 - \alpha^4 r_0^2, \min \{ t_0 + \alpha^4 r_0^2, T \} )$ we have
\[ d_t (x_0, y_0) > \alpha d_{t_0} (x_0, y_0). \]
The theorem then follows by replacing $\alpha$ by $\alpha^2$.
Assume that our claim was wrong.
By continuity of the distance function $d_t (x_0, y_0)$ in time, we can then find a time $t \in [t_0 - \alpha^4 r_0^2, \min \{ t_0 + \alpha^4 r_0^2, T \} )$ such that
\begin{equation} \label{eq:assumeequality}
 d_t (x_0, y_0) = \alpha d_{t_0} (x_0, y_0).
\end{equation}
We now apply the upper bound for $r_0 \leftarrow \alpha r_0$, $t_0 \leftarrow t$ and $t \leftarrow t_0$ (here and in the rest of the paper ``$\leftarrow$'' denotes replacement).
For this note that $|t - t_0| < \alpha^4 r_0^2 = \alpha^2 (\alpha r_0)^2$.
We obtain that
\[ d_{t_0} (x_0, y_0) < \alpha^{-1} d_{t} (x_0, y_0)) = d_{t_0} (x_0, y_0), \]
which contradicts (\ref{eq:assumeequality}).
This finishes the proof of the theorem.
\end{proof}

We now explain why the statement of Theorem \ref{thdistcon} becomes false if we don't impose the upper bound on the scalar curvature:
Since the Bryant soliton $(\M_{\textnormal{Bry}}, (g_{\textnormal{Bry}, t})_{t \in (-\infty, \infty)})$ arises as a model of certain type-II singularities, it suffices to show that the statement of the Theorem \ref{thdistcon}, without the scalar curvature bound, is false for the Bryant soliton.
Let $x_0 \in \M_{\textnormal{Bry}}$ be the tip of the Bryant soliton and choose an arbitrary point $y_0 \in\M_{\textnormal{Bry}}$, $y_0 \neq x_0$ such that $r_0 := d_0 (x_0, y_0) > 1$.
So $d_t (x_0,y_0) \geq r_0$ for $t \leq 0$ and, since the Ricci curvature is strictly positive around $x_0$, we have $\partial_t d_t (x_0,y_0) < -c$ if $t \leq 0$, for some universal $c > 0$, which does not depend on $r_0$. (Note that $\partial_t d_t (x_0,y_0)$ is equal to the integral of $-\Ric (\gamma', \gamma')$ along a minimizing geodesic $\gamma$ between $x_0, y_0$. The dimension of this integral is the inverse of length.)
So $d_t (x_0,y_0) / d_0 (x_0,y_0) > c |t|$ for $t \leq 1$, giving us a contradiction for large enough $|t|$ and $r_0$.

Next, we present a proof of Corollary \ref{coBDP}.

\begin{proof}[Proof of Corollary \ref{coBDP}]
Set $A = \max \{ \alpha^{-1/2}, \alpha^{-1} \}$, where $\alpha$ is the constant from Theorem \ref{thdistcon}.
In the following we will show the inequality for the case $t > t_0$.
The case $t < t_0$ follows analogously.
Assume that the inequality was wrong for some $r_0 > 0$, $t_0 \in [r_0^2, T)$ and $x_0, y_0 \in \M$ with $d_{t_0} (x_0, y_0) = r_0$.
Then we can choose $t_1 \in [t_0, T)$ such that
\[ d_{t_1} (x_0, y_0) = A \sqrt{r_0^2 + | t_1 - t_0 |} \leq \min \{ R_0^{-1/2}, \sqrt{t_0} \}. \]
Applying Theorem \ref{thdistcon} with $t_0 \leftarrow t_1$, $r_0 \leftarrow A \sqrt{r_0^2 + |t_1 - t_0|}$, $x_0 \leftarrow x_0$, $y_0 \leftarrow y_0$ yields
\begin{equation} \label{eq:distdistortionatotherscale}
 d_t (x_0, y_0) > \alpha A \sqrt{r_0^2 + | t_1 - t_0 |} \qquad \text{for all} \qquad t \in [t_1 - \alpha A^2 (r_0^2 + |t_1 - t_0|), t_1].
\end{equation}
By the choice of $A$ we have
\[ t_1 - \alpha A^2 (r_0^2 + |t_1 - t_0| ) \leq t_1 - (r_0^2 + |t_1 - t_0| ) < t_0. \]
So $t_0$ is included in the time-interval in (\ref{eq:distdistortionatotherscale}).
This implies that
\[ d_{t_0} (x_0, y_0) \geq \sqrt{r_0^2 + |t_1 - t_0|} > r_0, \]
which contradicts our assumptions.
\end{proof}

Using similar techniques, we construct a space-time cutoff function with bounded gradient, time-derivative and Laplacian.

\begin{proof}[Proof of Theorem \ref{Thm:cutoffnoLp}]
As explained in section \ref{sec:Preliminaries}, we may assume by parabolic rescaling that without loss of generality we have $r_0 = 1$ and $R \leq 1$ on $P(x_0, t_0, 1, -\tau)$  and $t_0 \geq 1$.
Moreover, we have $R \geq -n$ on $\M \times [t_0 - \frac12, t_0]$.

Let $0 < \theta < \frac12$ be a constant whose value will be determined in the course of the proof.
The constant $\rho$ will also be chosen in the course of the proof in such a way that $\rho^2 < \tfrac12 \theta$.
Since $0 < \tau \leq \rho^2$, we may assume in the following that we have $\tau < \frac12 \theta$.

Again, by \cite[end of 7.1]{P:1}, we can find a $z \in \M$ such that the reduced distance between $(x_0, t_0)$ and $(z, t_0 - \theta)$ satisfies
\[ l_{(x_0, t_0)} (z, t_0 - \theta ) \leq  \frac{n}2. \]
Consider the forward heat kernel $K(x,t) = K(x,t; z, t_0 - \theta)$ centered at $(z, t_0 - \theta)$, i.e. $\partial_t K = \Delta K$.
As in (\ref{eq:lowerbound}) we have for all $t \in (t_0 - \theta, t_0]$
\[
 K(\cdot, t) < \frac{B_0}{(t-(t_0 - \theta))^{n/2}} \qquad \text{on} \qquad \M
\]
and
\begin{equation} \label{eq:kboundsummarylower}
K(x_0, t_0) \geq \frac1{(4\pi ( t_0 - (t_0 - \theta))^{n/2}} e^{-l_{(x_0, t_0)} (z, t_0 - \theta)} \\
\geq  \frac1{(4\pi  \theta)^{n/2}} e^{-n/2} > B_0^{-1} \theta^{-n/2}.
\end{equation}
As in (\ref{eq:nabkdtk2bound}), we can apply Lemma \ref{leddu}(a) to the time-interval $[t_0 -  \frac34 \theta, t_0]$, and find a uniform constant $B < \infty$ such that
\begin{multline} \label{eq:dkdtk2nd}
 K < B \theta^{-n/2} \quad \text{on} \quad \M \times [t_0 - \tfrac12 \theta, t_0] \qquad \text{and} \\ \qquad |\partial_t K| < B \theta^{-n/2-1} \quad  \text{on} \quad P(x_0, t_0, 1, -\tau) \subset \M \times [t_0 - \tfrac12 \theta, t_0] .
\end{multline}
We now define 
\[ U := \Big\{ x \in \M \;\; : \;\; K(x,t_0) > \tfrac12 B_0^{-1} \theta^{-n/2} \Big\} \subset \M. \]
Recall that by (\ref{eq:kboundsummarylower}), we have $x_0 \in U$.
Let $U_0 \subset U$ be the connected component of $U$ that contains $x_0$.

We now claim that for the right choice of $\theta$ we have $U_0 \subset B(x_0, t_0, 1)$.
Assume not and let $\gamma : [0,1] \to U_0$ be a curve within $U_0$ that connects $x_0$ with a point on $\partial B(x_0, t_0, 1) \cap U_0$.
Let $N \geq 1$ be maximal with the property that we can find parameters $s_1, \ldots, s_N \in [0,1]$ such that the balls $B(\gamma(s_1), t_0, \sqrt{\theta}), \ldots, B(\gamma(s_N), t_0, \sqrt{\theta})$ are pairwise disjoint.
This implies that the balls of twice the radius, $B(\gamma(s_1), t_0, 2\sqrt{\theta}), \ldots, B(\gamma(s_N), t_0, 2\sqrt{\theta})$, cover $\gamma([0,1])$.
We now argue as in the proof of Theorem \ref{thdistcon}:
Consider a graph on $N$ vertices, $1, \ldots, N$, and connect vertices $i,j$ with an edge if $B(\gamma(s_i), t_0, 2 \sqrt{\theta})$ and $B(\gamma(s_j), t_0, 2 \sqrt{\theta})$ intersect.
This graph is connected and hence it is possible to connect the vertex corresponding to the ball that contains $x_0$ with the vertex corresponding to the ball that intersects $\partial B(x_0, t_0, 1)$ by a chain of length $\leq N-1$.
So we have
\begin{equation} \label{eq:lowerboundonN}
  2 \sqrt{\theta} + 4 \sqrt{\theta} (N-1) + 2 \sqrt{\theta} > 1 \qquad \Longrightarrow \qquad N > \frac{1}{4 \sqrt{\theta}}.
\end{equation}

On the other hand, we can use (\ref{eq:derivativeboundhk}) at time $t_0$, with $\frac12$ replaced by $\theta$ and $B$ replaced by $B \theta^{-n/2}$
\[  \Bigg| \nabla \sqrt{ \log  \frac{B \theta^{-n/2}}{K(\cdot, t_0)}  } \Bigg| < \frac{1}{2\sqrt{\theta}}. \]
So for each $y \in U$ and $z \in B(y, t_0, \sqrt{\theta})$
\[ \sqrt{ \log \bigg( \frac{B \theta^{-n/2}}{K(z, t_0)} \bigg) }  <  \frac12 + \sqrt{\log \bigg( \frac{B \theta^{-n/2}}{K(y, t_0)} \Big) } < \frac12 + \sqrt{\log \bigg( \frac{B \theta^{-n/2}}{B_0^{-1} \theta^{-n/2}} \Big) } = \frac12 + \sqrt{\log (B B_0) }, \]
which implies
\[ K(z,t_0)  > \frac{c}{\theta^{n/2}}, \]
for some uniform $c > 0$.
Therefore, using (\ref{eq:L1boundk}) and (\ref{eq:lowerboundonN})
\[ e^n > \int_{\M} K(\cdot, t_0) dg_{t_0} > \sum_{i =1}^N \int_{B(\gamma(s_i), t_0, \sqrt{\theta})} K(\cdot, t_0) dg_{t_0} > N \cdot \frac{c}{\theta^{n/2}} \cdot \kappa_1 \theta^{n/2} = N c  > \frac{c }{4 \sqrt{\theta}}. \]
So for $\theta < \frac1{16} e^{-2n} c^2$ we obtain a contradiction.
This implies that $U_0 \subset B(x_0, t_0, 1)$ if $\theta > 0$ is chosen small, but universally.
We will fix $\theta$ for the rest of the proof.

Assume now that $0 < \tau < 0.1 B^{-1} B_0^{-1} \theta < \frac12 \theta$.
Integrating (\ref{eq:dkdtk2nd}), we find that for all $(x,t) \in U_0 \times [t_0 - \tau, t_0]$ we have
\[ K(x, t_0 ) - 0.1 B_0^{-1} \theta^{-n/2} < K(x,t) < K(x, t_0 ) + 0.1 B_0^{-1} \theta^{-n/2} . \]
In particular, for $x \in \partial U_0 \times [t_0 - \tau, t_0]$ we have
\[ K(x,t) < 0.6 B_0^{-1} \theta^{-n/2}. \]
So if we define $\psi \in C^0 (U_0 \times [t_0 - \tau, t_0])$ by
\[ \psi (x,t) := \max \big\{ K(x,t) - 0.6 B_0^{-1} \theta^{-n/2}, 0 \big\}, \]
then $\supp \psi (\cdot, t)$ is compactly contained in $U_0$ for all $t \in [t_0 - \tau, t_0]$.
So $\psi$ can be continued to a function $\widehat{\psi} \in C^0 (\M \times [t_0 - \tau, t_0])$ by setting $\widehat{\psi} (x,t) := 0$ for $x \in \M \setminus U_0$.

Let us summarize the properties of $\widehat{\psi}$.
First, we have
\[ 0 \leq \widehat{\psi} < B \theta^{-n/2}. \]
on $\M \times [t_0 - \tau, t_0]$.
We also have
\[ |\nabla \widehat{\psi} | < B \theta^{-n/2 - 1/2}, \qquad |\partial_t \widehat{\psi}| + |\Delta \widehat{\psi} | < 2 B \theta^{-n/2 - 1} \]
whenever $\widehat{\psi} > 0$.
Lastly, we have
\[ \widehat{\psi} (x_0, t_0) > 0.4 B_0^{-1} \theta^{-n/2}. \]
So if we set $\widetilde{\phi} := \frac1{16} B^{-1} \tau^{n/2 +1}  \widehat{\psi}$, then
\[ 0 \leq \widetilde{\phi} < 1 \qquad \text{and} \qquad \widetilde{\phi} (x_0, t_0 ) > c  \]
for some uniform constant $c > 0$ and
\[  |\nabla \widetilde{\phi} | < \tfrac14, \qquad |\partial_t \widetilde{\phi} | + | \Delta \widetilde{\phi} | < \tfrac18, \]
whenever $\widetilde{\phi} > 0$.
So if we set $\phi := \widetilde{\phi}^2$, then assertions (a), (c) and (d) hold for an appropriate $\rho$.
Assertion (b) follows from the fact that $\phi (x_0, t_0) > c^2$ and assertion (d).
Note that $\phi$ is twice differentiable and can be mollified to become smooth while preserving assertions (a)--(d).
\end{proof}

As an application of Theorem \ref{Thm:cutoffnoLp}, we construct a cutoff function with bounded time derivative, which vanishes outside of a ball at one time level and whose time-derivative is positive and bounded away from zero.
This function will be used as a barrier in the proof of Theorem \ref{Thm:backwpseudoloc} in section \ref{sec:pseudoloc}.

\begin{lemma} \label{Lem:secondcutoff} 
Let $(\M^n, (g_t)_{t \in [0,T)})$, $T < \infty$ be a Ricci flow on a compact $n$-manifold.
Then there is a constant $\rho > 0$, which only depends on $\nu[g_0, 2T], n$, such that the following holds:

Let $(x_0, t_0) \in \M \times [0,T)$ and $0 < r_0 \leq  \sqrt{t_0}$ and $0 < \tau \leq 2 \rho^2 r_0^2$.
Assume that $R \leq r_0^{-2}$ on $P(x_0, t_0, r_0, -\tau)$.
Then there exists a cutoff function $\phi \in C^{\infty}({\M} \times [t_0 - \tau, t_0])$
with the following properties:
\begin{enumerate}[label = (\alph*)]
\item $0 \le \phi < r_0$ everywhere.
\item $\phi > \rho r_0$ on $P(x_0, t_0, \rho r_0, -\tau)$.
\item $\phi (\cdot, t_0) = 0$ on $(\M \setminus B(x_0, t_0, r_0)) \times [t_0 - \tau, t_0]$.
\item $\phi (\cdot, t)$ is $1$-Lipschitz for each $t \in [t_0 - \tau, t_0]$.
\item $\p_t \phi > r^{-1}_0$ on $P(x_0, t_0, \rho r_0, -\tau)$.
\end{enumerate}
\end{lemma}

\begin{proof}
Let $\widetilde{\phi} \in C^\infty (\M \times [t_0 - \tau, t_0])$ be the cutoff function from Theorem \ref{Thm:cutoffnoLp} and set
\[ \phi (x,t) := r_0 \max \big\{  \widetilde{\phi} (x,t) - 2 r_0^{-2} (t_0 - t) , 0 \big\}. \]
Assertions (a), (c), (d) are obviously true.
Let now $(x,t) \in P(x_0, t_0, \rho r_0, - \tau)$.
Then, assuming $\rho < 0.1$, we have
\[ \phi (x,t) > r_0 ( \rho - 2\rho^2 ) > \tfrac12 \rho r_0 \]
and
\[ \partial_t \phi (x,t) > r_0 ( - r_0^{-2} + 2 r_0^{-2} ) = r_0^{-1}. \]
So assertions (b) and (e) hold after replacing $\rho$ by $\frac12 \rho$.
\end{proof}

\section{Mean value inequalities for the heat and conjugate heat equation} \label{sec:meanvalue}
In this section we present some consequences of the existence of the cutoff function from Theorem \ref{Thm:cutoffnoLp}.
Our first result, Lemma \ref{leCHEmv}, will be a mean value inequality for solutions of the conjugate heat equation. 
This mean value inequality will be used in the subsequent section to deduce a Gaussian upper bound for heat kernels.
Similar mean value inequalities have been proven under various extra conditions, such as boundedness
of $\partial_t g$.
Here we just assume a global bound on the scalar curvature.
Note that the results of this section can be generalized easily to the setting in which the scalar curvature is only locally bounded.

\begin{lemma} \lab{lemoserCHE}
Let $(\M^n, (g_t)_{t \in [0,T)})$, $T < \infty$ be a Ricci flow on a compact manifold that satisfies $R \leq R_0 < \infty$ everywhere and let $p > 2$.
Then there are constants $0 < \alpha < 1$, $A < \infty$ that only depend on $p, \nu[g_0, 2T], n$ such that the following holds:

Let $(x_0, t_0) \in \M \times [0,t)$ and $0 < r_0 \leq \min \{ R_0^{-1/2}, \sqrt{t_0} \}$ and let $u \in C^\infty ( \M \times [t_0, t_0 + r_0^2])$ be a solution of the conjugate heat equation (\ref{che}).
Then
\[
\bigg( \int_{Q^+ (x_0, t_0, \a r_0)} u^p dg_t dt \bigg)^{2/p}
\le \frac{A}{r_0^{(n+2) (p-2)/p}} \int _{Q^+ (x_0, t_0,  r_0 )} u^2 dg_tdt.
\]
\end{lemma}

\begin{proof}
We may assume by parabolic rescaling that $r_0 = 1$ and hence $|R|  \leq n$ on $\M \times [t_0 - \tfrac12, t_0]$ (see section \ref{sec:Preliminaries} for more details).
Let $\rho > 0$ be the constant from Theorem \ref{Thm:cutoffnoLp}.
By the distance distortion estimate of Theorem \ref{thdistcon}, we can find some uniform constant $0 < \beta < 1$ such that
\[  B(x_0, t_1 , \beta) \subset B(x_0, t_2, 1) \qquad \text{for all} \qquad t_1, t_2 \in [t_0, t_0 + \beta^2 \rho^2]. \]
Next, we can find some uniform constant $0 < \gamma < \beta \rho$ such that
\[  B(x_0, t_1 , \gamma) \subset B(x_0, t_2, \beta \rho) \qquad \text{for all} \qquad t_1, t_2 \in [t_0, t_0 + 2 \gamma^2]. \]

Let $\phi \in C^\infty (\M \times [t_0, t_0 + 2\gamma^2])$ be the cutoff function from Theorem \ref{Thm:cutoffnoLp} applied at scale $\beta$ and at time $t_0 + 2 \gamma^2$.
Recall that for some uniform constant $C_0 < \infty$:
\begin{enumerate}[label=(\roman*)]
\item $0 \leq \phi < 1$, $|\nabla \phi | < \beta^{-1} \leq C_0$ and $|\partial_t \phi | < \beta^{-2} \leq C_0$ on $\M \times [t_0, t_0 + 2 \gamma^2]$.
\item $\phi = 0$ on $(\M \setminus B(x_0, t_0 + 2 \gamma^2, \beta)) \times [t_0, t_0 + 2 \gamma^2 ] \supset \M \times [t_0, t_0 + 2 \gamma^2] \setminus Q^+ (x_0, t_0, 1)$.
\item $\phi > \rho$ on $P(x_0, t_0 + 2 \gamma^2, \beta \rho, -2 \gamma^2) \supset Q^+ (x_0, t_0, \gamma  )$.
\end{enumerate}
Let $\eta \in C^\infty ([t_0, t_0 + 2\gamma^2])$ be a cutoff function with $\eta \equiv 1$ on $[t_0, t_0 + \gamma^2]$, $\eta (t_0 + 2 \gamma^2) = 0$, $\eta \leq 1$ and $|\eta' | < C_1$ for some uniform constant $C_1 < \infty$.
Then the function
\[ \psi (x,t) := \eta (t) \phi (x,t) \]
still satisfies the properties (i), (ii) above, possibly after adjusting $C_0$ and we have $\psi > \rho$ on $Q^+ (x_0, t_0, \gamma)$.

The remainder of our proof resembles the standard Moser iteration except that we need to use
the improved cut off function $\psi$.
We first note that
\begin{equation} \label{eq:backwardtothep}
 \Delta u^{p/2} - \frac{p}2 R u^{p/2} + \partial_t u^{p/2} \geq 0.
\end{equation}
So integrating $u^p$ against $\psi^2$ at a some time $t \in [t_0, t_0 + 2 \gamma^2]$ yields, for a generic constant $C$ and for $B := B(x_0, t_0 + 2 \gamma^2, \beta)$
\begin{alignat*}{1}
 \frac{d}{dt} \int_{\M} u^p \psi^2 dg_t &= 2 \int_{\M} (\partial_t u^{p/2}) u^{p/2} \psi^2 dg_t + 2 \int_{\M} u^p (\partial_t \psi) \psi dg_t -  \int_{\M} u^p \psi^2 R dg_t \displaybreak[1] \\
 &\geq 2 \int_{\M} \Big(- \Delta u^{p/2} + \frac{p}2 R u^{p/2} \Big) u^{p/2} \psi^2 dg_t - C \int_{B} u^p dg_t \displaybreak[1] \\
 &\geq 2 \int_{\M} \big(\nabla u^{p/2} \big) \cdot \big( \nabla (u^{p/2} \psi) \psi + u^{p/2} \psi \nabla \psi \big) dg_t - C  \int_{B} u^p dg_t \displaybreak[1] \\
 &= 2 \int_{\M} \big(\nabla (u^{p/2} \psi) - u^{p/2} \nabla \psi \big) \cdot \big( \nabla (u^{p/2} \psi)  + u^{p/2} \nabla \psi \big) dg_t - C  \int_{B} u^p dg_t \displaybreak[1] \\
  &= 2 \int_{\M} \big|\nabla (u^{p/2} \psi ) \big|^2 dg_t - 2 \int_{\M}  u^{p} |\nabla \psi|^2 dg_t - C \int_{B} u^p dg_t \displaybreak[1] \\
  &\geq  2 \int_{\M} \big|\nabla (u^{p/2} \psi ) \big|^2 dg_t - C \int_{B} u^p dg_t .
 \end{alignat*}
By \cite{Z07:1} (see also \cite{Ye:1}), there is a Sobolev inequality of the form
\begin{equation} \label{eq:Sobolevcite}
 \bigg( \int_{\M} v^{\frac{2n}{n-2}} dg_t \bigg)^{\frac{n-2}{2n}} \leq C_S \bigg( \int_{\M} |\nabla_t v|^2 dg_t \bigg)^{1/2} + C_S \bigg( \int_{\M} v^2 dg_t \bigg)^{1/2},
\end{equation}
for any $v \in C^1 (\M)$ and $t \in [0,T)$, where $C_S < \infty$ only depends on $\nu[g_0, 2T]$ and $n$.
So we obtain
\[ \frac{d}{dt} \int_{\M} u^p \psi^2 dg_t \geq C^{-1} \bigg( \int_{\M} u^{\frac{p n}{n-2}} \psi^2 dg_t \bigg)^{\frac{n-2}{n}} - C \int_{B} u^p dg_t.  \]
Integrating this inequality from $t_1 \in [t_0, t_0 + 2 \gamma^2]$ to $t_0 + 2 \gamma^2$ and observing that the integral on the left-hand side vanishes for $t = t_0 + 2 \gamma^2$ yields
\begin{multline*}
0 - \int_{\M} u^p(\cdot, t_1) \psi^2(\cdot, t_1) dg_{t_1}  \\
  \geq C^{-1} \int_{t_1}^{t_0 + 2\gamma^2}  \bigg( \int_{\M} u^{\frac{pn}{n-2}} \psi^2 dg_t \bigg)^{\frac{n-2}n} dt - C \int_{t_1}^{t_0 +  2 \gamma^2} \int_B u^p dg_t dt.
\end{multline*}
Letting $t_1$ vary over $[t_0, t_0 + 2\gamma^2]$, we obtain
\begin{multline*}
 \sup_{t \in [t_0, t_0 + 2\gamma^2]} \int u^p (\cdot, t) \psi^2 ( \cdot, t) dg_t +  \int_{t_0}^{t_0 + 2\gamma^2}  \bigg( \int_{\M} u^{\frac{pn}{n-2}} \psi^2 dg_t \bigg)^{\frac{n-2}n} dt \\ 
 \leq C  \int_{t_0}^{t_0+ 2\gamma^2} \int_Bu^p dg_t dt.
\end{multline*}
So by H\"older's inequality
\begin{multline*}
 \int_{t_0}^{t_0 + 2 \gamma^2}  \int_{\M} u^{p( 1 + \frac{2}n)} \psi^2 dg_t dt 
= \int_{t_0}^{t_0 + 2 \gamma^2}  \int_{\M} u^{p} \psi^{\frac{2(n-2)}{n}} \cdot u^{\frac{2}n p} \psi^{\frac{4}n} dg_t dt \\
\leq  \int_{t_0}^{t_0 + 2 \gamma^2} \bigg( \int_{\M} u^{\frac{pn}{n-2}} \psi^2 dg_t \bigg)^{\frac{n-2}{n}} \bigg( \int_{\M} u^p \psi^2 dg_t \bigg)^{\frac{2}n} dt \\
\leq  \bigg( \int_{t_0}^{t_0 + 2 \gamma^2} \bigg( \int_{\M} u^{\frac{pn}{n-2}} \psi^2 dg_t \bigg)^{\frac{n-2}{n}} dt \bigg) \cdot \bigg(C  \int_{t_0}^{t_0+ 2\gamma^2} \int_Bu^p dg_t dt \bigg)^{\frac{2}n}  \\
 \leq  \bigg( C \int_{t_0}^{t_0+ 2 \gamma^2} \int_B u^p dg_t dt \bigg)^{1 + \frac{2}n}.
\end{multline*}
Using the fact that $\psi > \rho$ on $Q^+(x_0, t_0, \gamma)$, it follows that
\[ \bigg( \int_{Q^+ (x_0, t_0, \gamma )} u^{p( 1 + \frac{2}n)} dg_t dt \bigg)^{1 / p( 1 + \frac{2}n)} \leq C \bigg( \int_{Q^+ (x_0, t_0, 1)} u^p dg_t dt \bigg)^{1/p}. \]
Applying this inequality successively at scales $\gamma^k$, yields the desired inequality for $p = 2(1 + \frac{2}n )^k$ and $\alpha = \gamma^k$ and some $A$, which may depend on $k$.
If $p$ takes other values, we can use interpolation to prove the result.
\end{proof}

Next we prove a mean value inequality for the conjugate heat
equation. At  first glance, it may seem that the Moser iteration in the proof of the previous lemma
will lead to the mean value inequality. However, since the size of the cube
$Q^+ (x_0, t_0, \gamma^k )$ shrinks to $0$ when $k \to \infty$, we have to take a different path.

\begin{lemma} \label{leCHEmv}
Let $(\M^n, (g_t)_{t \in [0,T)})$, $T < \infty$ be a Ricci flow on a compact manifold that satisfies $R \leq R_0 < \infty$ everywhere.
Then there are constants $\beta > 0$, $C < \infty$, which only depend on $\nu[g_0, 2T], n$, such that the following holds:

Let $(x_0, t_0) \in {\M} \times [0, T)$ and $0 < r_0 \leq \min \{ R_0^{-1/2}, \sqrt{t_0} \}$ and let $u \in C^\infty ( \M \times [t_0, t_0 + r_0^2])$ be a solution of the conjugate heat equation (\ref{che}).
Then
\[
\sup_{Q^+(x_0, t_0, \beta r_0)} u^2 \le \frac{C}{r_0^{n+2}}
\int_{Q^+(x_0, t_0, r_0)} u^2 dg_t dt.
\]
\end{lemma}

\begin{proof}
As in the previous proof, we may assume that $r_0 = 1$ and $|R|  \leq n$ on $\M \times [t_0 - \frac12, t_0]$.
Let $\alpha  > 0$ be the constant from Lemma \ref{lemoserCHE} for some $p > n+2$, which we will fix for the rest of the proof and let $\rho > 0$ be the constant from Theorem \ref{Thm:cutoffnoLp}.
By Theorem \ref{thdistcon} there is a uniform constant $\gamma > 0$ such that
\[ P(x_0, s, \gamma \rho , s-t_0) \subset Q^+ (x_0, t_0, \alpha) \qquad \text{for all} \qquad s \in [t_0, t_0 + (\gamma \rho)^2]. \]
Next, we can find a uniform constant $0 < \beta < (\gamma \rho)^2$ such that
\[ Q^+ (x_0, t_0, \beta) \subset P(x_0, t_0 + 2\beta^2, \gamma \rho, - 2\beta^2). \]
So we have
\begin{equation} \label{eq:PinOmega}
 Q^+ (x_0, t_0, \beta) \subset P(x_0, t_0 + 2 \beta^2, \gamma \rho, - 2 \beta^2) \subset Q^+ (x_0, t_0, \alpha ).
 \end{equation}
Apply Theorem \ref{Thm:cutoffnoLp} at the point $x_0$, at scale $\gamma$ and at time $t_0 + 2\beta^2$.
We obtain a cutoff function $\phi \in C^\infty ( \M \times [t_0, t_0 + 2 \beta^2])$ with the following properties for some universal constant $C_0 < \infty$:
\begin{enumerate}[label=(\roman*)]
\item $0 \leq \phi < 1$, $| \nabla \phi (\cdot, t) | \leq \gamma^{-1} \leq C_0$, $|\partial_t \phi| + |\Delta \phi | \leq \gamma^{-2} \leq C_0$ on $\M \times [t_0, t_0 + 2\beta^2]$.
\item $\phi = 0$ on $(\M \setminus B(x_0, t_0 + 2 \beta^2, \gamma)) \times [t_0, t_0 + 2\beta^2]$.
\item $\phi > \rho$ on $P(x_0, t_0 + 2\beta^2, \gamma \rho, - 2\beta^2)$.
\end{enumerate}
Let $\eta \in C^\infty ([t_0, t_0 + 2\beta^2])$ be a cutoff function with $\eta \equiv 1$ on $[t_0, t_0 + \beta^2]$, $\eta (t_0 +2\beta^2) = 0$, $\eta \leq 1$ and $|\eta' | < C_1$ for some uniform constant $C_1 < \infty$.
Then the function
\[ \psi (x,t) := \eta (t) \phi (x,t) \]
still satisfies properties (i), (ii) above, possibly with a larger constant $C_0$ and we have $\psi (\cdot, t_0 +2\beta^2) \equiv 0$.

Notice that
\[ \Delta (u \psi) - R u \psi + \partial_t (u\psi) = u \Delta \psi + u \partial_t \psi + 2 \nabla u \nabla \psi. \]
Fix $(x, t) \in Q^+ (x_0, t_0, \beta)$ and consider the heat kernel $K (z, s) = K(z, s ; x, t)$ of the forward heat equation, $\partial_t K = \Delta K$, centered at $(x,t)$.
Then
\[  (u \psi) (x,t) = \int_t^{t_0 + 2\beta^2} \int_{\M} K(z,s) \big( u \Delta \psi + u \partial_s \psi + 2 \nabla u \nabla \psi \big) (z,s) dg_s (z) ds.  \]
After integration by parts, this becomes
\begin{multline*}
 (u \psi) (x,t) =   \int_t^{t_0 + 2\beta^2} \int_{\M} K(z,s) u(z,s) \big( - \Delta \psi + \partial_s \psi) (z,s) dg_s(z) ds  \\
+ 2  \int_t^{t_0 + 2\beta^2} \int_{\M} u(z,s) \nabla_z K(z,s) \nabla \psi (z,t) dg_s(z) ds
\end{multline*}
From properties (i), (ii) of the cutoff function $\psi$ and (\ref{eq:PinOmega}), it follows that
\begin{alignat*}{1}
|u \psi |(x,t) &\leq C_0 \int_{t}^{t_0 + 2\beta^2} \int_{B(x_0, s, \alpha)} K(z,s) |u| (z,s) dg_s (z) ds \\
&\qquad\qquad + 2C_0 \int_t^{t_0 + 2\beta^2} \int_{B(x_0, s, \alpha)} |\nabla_z K(z,s)| \cdot |u| (z,s) dg_s (z) ds \\
&\equiv C_0 I_1 + 2 C_0 I_2.
\end{alignat*}
Next we bound $I_1$.
By H\"older's inequality, for $q = p/(p-1)$,
\begin{alignat*}{1}
I_1 &\leq \bigg( \int_t^{t_0 + 2\beta^2} \int_{B(x_0, s, \alpha)} K^q(z,s) dg_s (z) ds \bigg)^{1/q} \bigg( \int_t^{t_0 + 2\beta^2} \int_{B(x_0, s, \alpha)} |u|^p dg_s ds \bigg)^{1/p} \\
&\leq \bigg( \int_t^{t_0 + 2\beta^2} \int_{B(x_0, s, \alpha)} K^q(z,s) dg_s (z) ds \bigg)^{1/q} A \Vert u \Vert_{L^2 (Q^+(x_0, t_0, 1))},
\end{alignat*}
where we have used Lemma \ref{lemoserCHE}.
By (\ref{eq:Prel-upperhkbound}), there exists a universal constant $C$ such that
\begin{equation} \label{G<1/tn/2}
K(z, s) \le \frac{C}{(s-t)^{n/2}} \qquad \text{for} \qquad s \in (t, t + 1]. 
\end{equation}
On the other hand, we deduce from
\[ \frac{d}{ds} \int_{\M} K(z,s) dg_s (z) = - \int_{\M} R(z,s) K(z,s) dg_s (z) \leq n \int_{\M} K(z,s) dg_s (z) \]
that
\[ \int_{\M} K(z,s ) dg_s (z) \leq C \qquad \text{for} \qquad s \in (t, t+1]. \]
So by the interpolation inequality
\begin{alignat*}{1}
I_1 &\leq C \bigg( \int_{t}^{t_0 + 2\beta^2} \int_{B(x_0, s, \alpha)} K (z, s)  \frac1{(s-t)^{(q-1) n/2}} dg_s (z) ds \bigg)^{1/q} A \Vert u \Vert_{L^2 (Q^+(x_0, t_0, 1))} \\
 &\leq C \bigg( \int_t^{t_0+ 2\beta^2}  \frac1{(s-t)^{(q-1) n/2}} ds \bigg)^{1/q} A \Vert u \Vert_{L^2 (Q^+(x_0, t_0, 1))} .
\end{alignat*}
Since $p > n+2$, we have $(q - 1) n/2 < 1$.
This shows that
\[
I_1 \leq C A \Vert u \Vert_{L^2(Q^+(x_0, t_0, 1))}.
\]

Next we bound $I_2$, which requires a bound on $\nabla_z K(z,s ; x, t)$.
From (\ref{G<1/tn/2}), we know that for all $s \in (t, t+1]$.
\[ J_s := \sup_{\M \times [\frac{s+t}2,s]} K (\cdot, \cdot ; x, t) \leq \frac{2^{n/2} C}{(s-t)^{n/2}}. \]
Since the function $K(z,s) = K(z,s; x,t)$ is a solution to the (forward) heat equation, we deduce using (\ref{eq:Prel-derbound})
\[ \frac{|\nabla_z K(z,s)|^2}{K^2 (z,s)} \leq \frac{C}{s-t} \log \frac{J_s}{K(z,s)}. \]
Letting $q = p /(p-1)$ again, we find
\[ |\nabla_z K(z,s) |^q \leq K (z,s) \frac{C J_s^{q-1}}{(s-t)^{q/2}} \cdot \bigg( \frac{K(z,s)}{J_s} \bigg)^{q-1} \bigg( {\log \frac{J_s}{K(z,s)}} \bigg)^{q/2}. \]
Note that since $K(z,s) / J_s \leq 1$, the product of the last two factors is bounded from above, which implies
\[ |\nabla_z K(z,s) |^q \leq \frac{C}{(s-t)^{((n+1)q - n)/2}} K(z,s), \]
where $C$ only depends on $\nu[g_0, 2T]$ and $n$.
Using this gradient bound and H\"older's inequality, we deduce similarly as before
\begin{alignat*}{1}
I_2 &\leq \bigg( \int_t^{t_0 + 2\beta^2} \int_{B(x_0, s, \alpha)} |\nabla_z K|^q (z,s) dg_s (z) ds \bigg)^{1/q} \bigg( \int_t^{t_0 +2\beta^2} \int_{B(x_0, s, \alpha)} |u|^p dg_s ds \bigg)^{1/p} \\
&\leq \bigg( \int_t^{t_0 + 2\beta^2} \int_{B(x_0, s, \alpha)} K (z,s) dg_s \frac{C}{(s-t)^{((n+1)q - n)/2}} ds \bigg)^{1/q} A \Vert u \Vert_{L^2 (Q^+ (x_0, t_0, 1))} \\
&\leq \bigg( \int_t^{t_0 + 2\beta^2}  \frac{C}{(s-t)^{((n+1)q - n)/2}} ds \bigg)^{1/q} A \Vert u \Vert_{L^2 (Q^+ (x_0, t_0, 1))} 
\end{alignat*}
From $p > n+2$, one knows $((n+1)q - n)/2 < 1$, therefore
\[ I_2 \leq CA \Vert u \Vert_{L^2 (Q^+ (x_0, t_0, 1))}. \]

Altogether, this gives
\[ |u\psi|(x,t) \leq C A \Vert u \Vert_{L^2 (Q^+ (x_0, t_0, 1))}. \]
The claim now follows from the fact that $\psi (x,t) > \rho$.
\end{proof}

Next we present  a mean value inequality for the (forward) heat equation.
We begin with a lemma that is similar to Lemma \ref{lemoserCHE} on the conjugate heat equation.

\begin{lemma}
\lab{lemoserHE}
Let $(\M^n, (g_t)_{t \in [0,T)})$, $T < \infty$ be a Ricci flow on a compact $n$-manifold that satisfies $R \leq R_0 < \infty$ everywhere and let $p > 2$.
Then there are constants $0 < \alpha < 1$, $A < \infty$, which only depend on $\nu[g_0, 2T]$, $n$ and $p$, such that the following holds:

Let $(x_0, t_0) \in \M \times [0,T)$ and $0 < r_0 \leq \min \{ R_0^{-1/2}, \sqrt{t_0} \}$ and let $u \in C^\infty (\M \times (t_0 - r_0^2, t_0])$ be a solution of the  heat equation (\ref{he}).
Then
\begin{equation} \label{eq:Lpboundforward}
\bigg( \int_{Q^- (x_0, t_0, \a r)} |u|^p dg_t dt \bigg)^{2/p}
\le \frac{A}{r^{(n+2) (p-2)/p}} \int _{Q^-(x_0, t_0,  r)} |u|^2 dg_t dt
\end{equation}
and
\begin{equation} \label{eq:nablaLpboundforward}
\bigg( \int_{Q^- (x_0, t_0, \a r)} |\d u|^p dg_t dt \bigg)^{2/p}
\le \frac{A}{r^{((n+2) (p-2)/p)+2}} \int _{Q^-(x_0, t_0,  r)} |u|^2 dg_t dt.
\end{equation}
\end{lemma}

\begin{proof}
The proof is similar to that of Lemma \ref{lemoserCHE}.
Assume again that $r_0 = 1$ and that $|R| \leq n$ on $\M \times [t_0 - \frac12, t_0]$.
Using Theorem \ref{Thm:cutoffnoLp} and the distance distortion estimates of Theorem \ref{thdistcon}, we can construct a cutoff function $\psi \in C^\infty ( \M \times [t_0 - 2 \gamma^2, t_0] )$, for some uniform $\gamma > 0$, such that for some uniform $C_0 < \infty$
\begin{enumerate}[label=(\roman*)]
\item $0 \leq \psi < 1$, $| \nabla \psi | < C_0$, $| \partial_t \psi | < C_0$ on $\M \times [t_0 - 2 \gamma^2, t_0]$.
\item $\psi = 0$ on $( \M \times [t_0 - 2 \gamma^2, t_0 ] ) \setminus Q^- (x_0, t_0, 1)$.
\item $\psi > \rho$ on $Q^- (x_0, t_0, \gamma)$.
\item $\psi (\cdot, t_0 - 2\gamma^2) \equiv 0$.
\end{enumerate}
Analogously to (\ref{eq:backwardtothep}), we use the inequality
\begin{equation} \label{eq:forwardtothep}
 \Delta u^{p/2} - \partial_t u^{p/2} \geq 0
\end{equation}
and integrate $u^p$ against $\psi^2$, which gives us for any $t \in [t_0 - 2 \gamma^2, t_0]$
\begin{equation} \label{eq:dtintegralforward}
 - \frac{d}{dt} \int_{\M} u^p \psi^2 dg_t dt \geq 2 \int_{\M} \big| \nabla ( u^{p/2} \psi ) \big|^2 dg_t - C \int_{B(x_0, t, 1)} u^p dg_t. 
\end{equation}
Applying the Sobolev inequality (\ref{eq:Sobolevcite}), integration in time and using H\"older's inequality and (iii) yields
\[ \bigg( \int_{Q^- (x_0, t_0, \gamma)} u^{p (1+ \frac2n)} dg_t dt \bigg)^{1/p(1+ \frac2n)} \leq C \bigg( \int_{Q^- (x_0, t_0, 1)} u^p dg_t dt \bigg)^{1/p}. \]
Iterating this inequality for $p = 2 (1+ \frac2n )^k$ and $\alpha = \gamma^k$ and using an interpolation argument proves (\ref{eq:Lpboundforward}).

In order to prove (\ref{eq:nablaLpboundforward}), we first integrate (\ref{eq:dtintegralforward}) for $p = 2$ in time and obtain
\[  - \int_{\M} u^2 (\cdot, t_0) \psi^2 (\cdot, t_0) dg_t dt \geq 2 \int_{t_0 - 2 \gamma^2}^{t_0} \int_{\M} \big | \nabla (u \psi) \big|^2 dg_t dt - C \int_{Q^- (x_0, t_0, 1)} u^2 dg_t dt. \]
So
\[ \int_{t_0 - 2 \gamma^2}^{t_0} \int_{\M} \big | \nabla (u \psi) \big|^2 dg_t dt  \leq C \int_{Q^- (x_0, t_0, 1)} u^2 dg_t dt. \]
Since $| \nabla u  |^2 \psi^2 = | \nabla (u\psi) - u \nabla \psi |^2 \leq 2 |\nabla ( u \psi) |^2 \psi^2 + 2 u^2 |\nabla \psi |^2$, we have
\[ \int_{Q^- (x_0, t_0, \gamma )} |\nabla u |^2 dg_t dt \leq C \int_{Q^- (x_0, t_0, 1)} u^2 dg_t dt. \]
We will now lift this $L^2$-estimate to an $L^p$-estimate using Moser iteration.
The proof is similar to the proof of (\ref{eq:Lpboundforward}), except that we have to replace $u$ by $|\nabla u|$ and (\ref{eq:forwardtothep}) by
\[ \Delta |\d u|^2 - \partial_t |\d u|^2 = 2 | {\Hess} \; u|^2 \geq 0. \]
Then we obtain
\[ \bigg( \int_{Q^- (x_0, t_0, \gamma^2)} |\nabla u|^{p (1+ \frac2n)} dg_t dt \bigg)^{1/p(1+ \frac2n)} \leq C \bigg( \int_{Q^- (x_0, t_0, 1)} |\nabla u|^p dg_t dt \bigg)^{1/p}. \]
Iterating this inequality proves the desired result.
\end{proof}

Now we state and prove the mean value inequality for the heat equation.

\begin{proposition} \label{prHEmv}
Let $(\M^n, (g_t)_{t \in [0,T)})$, $T < \infty$ be a Ricci flow on a compact $n$-manifold that satisfies
$R \leq R_0 < \infty$ everywhere.
Then there are constants $\beta > 0$, $C < \infty$, which only depend on $\nu[g_0, 2T], n$, such that the following holds:

Let $(x_0, t_0) \in \M \times [0,T)$ and $0 < r_0 \leq \min \{ R_0^{-1/2}, \sqrt{t_0} \}$ and let $u \in C^\infty (\M \times ( t_0 - r_0^2 , t_0])$ be a solution of the heat equation (\ref{he}).
Then
\[ \sup_{Q^-(x_0, t_0, \a r_0)} u^2 \le \frac{C}{r^{n+2}} \int_{Q^-(x_0, t_0, r_0)} u^2 dg_t dt. \]
\end{proposition}

\begin{proof}
The proof is similar to that of Lemma \ref{leCHEmv}.
Assume again that $r_0 = 1$ and hence that $| R | \leq n$.
Fix some $p > n+2$ and determine $\alpha > 0$ from the previous Lemma \ref{lemoserHE}.
Using Theorem \ref{Thm:cutoffnoLp} and the distance distortion estimates of Theorem \ref{thdistcon}, we can construct a cutoff function $\psi \in C^\infty ( \M \times [0, t_0])$ for some uniform $0 < \beta < \alpha$ such that
\begin{enumerate}[label=(\roman*)]
\item $0 \leq \psi < 1$, $|\nabla \psi(\cdot, t) | < C_0$, $|\partial_t \psi | + |\Delta \psi | < C_0$ on $\M \times [0, t_0]$
\item $\psi = 0$ on $\M \times [0, t_0) \setminus Q^- (x_0, t_0, \alpha )$.
\item $\psi > \rho$ on $Q^- (x_0, t_0, \beta)$.
\end{enumerate}
Clearly
\[ \Delta (u \psi) - \p_t (u \psi)= u( \Delta \psi - \partial_t \psi) + 2 \d \psi \d u. \] 
Fixing $(x, t) \in Q^- (x_0, t_0, \beta )$, let $K (y, s) = K(x, t; y, s)$ be the fundamental solution of the conjugate heat equation on the whole manifold.  
Then
\[
(u \psi) (x, t) = - \int_{0}^{t} \int_{\M} K(y, s) \big( u  (\Delta  - \partial_t ) \psi  +
2 \d \psi \d u \big) (y, s) dg_s (y) ds.
\]  
From properties (i)--(ii) of the cutoff function $\psi$,  we deduce
\[ |u \psi | (x, t) \le C   \int^t_{t_0-\a^2} \int_{B(x_0, s, \a)} K( y, s) \big( |u| + |\nabla u| \big) (s,t) dg_s(y) ds
\]
By H\"older's inequality, for
$q = p/(p-1)$,
\[
\al
|u \psi | (x,t) &\le  \Big(\int^t_{t_0-\a^2} \int_{B(x_0, s, \a)} K^q(y, s) dg_s (y) ds \Big)^{1/q} \times \\
&\qquad\qquad \bigg[ \bigg(\int^t_{t_0-\a^2 } \int_{B(x_0, s, \a)} |u|^p(y,s) dg_s (y) ds \bigg)^{1/p} \\
&\qquad\qquad\qquad + \bigg(\int^t_{t_0-\a^2 } \int_{B(x_0, s, \a)} |\nabla u|^p(y,s) dg_s (y) ds \bigg)^{1/p}  \bigg] \\
&\le \left(\int^t_{t_0-\a^2} \int_{B(x_0, s, \a)} K^q(y, s) dg_s (y) ds \right)^{1/q}  A \Vert u \Vert_{L^2(Q^-(x_0, t_0, 1))},
\eal
\]
where we have used Lemma \ref{lemoserHE}. 
By (\ref{eq:Prel-upperhkbound}) again, there exists a universal constant $C_1 < \infty$ such that
\[ K(y, s) \le \frac{C_1}{(t-s)^{n/2}} \qquad \text{for} \qquad s \in [t-1, t). \]
Substituting this to the previous inequality, we deduce
\[ |u \psi|(x,t) \le C \left(\int^t_{t_0-\a^2} \int_{B(x_0, s, \a)} K( y, s) dg_s (y) \frac{1}{(t-s)^{(q-1) n/2}} ds \right)^{1/q}  A \Vert u \Vert_{L^2(Q^-(x_0, t_0, 1))}. \] 
Since $\int_{B(x_0, s, \a)} K( \cdot , s) dg_s  \le 1$, this implies
\[ |u \psi| (x,t) \le C \left(\int^t_{t_0-\a^2} \frac{1}{(t-s)^{(q-1) n/2}} d\tau \right)^{1/q}  A \Vert u \Vert_{L^2(Q^-(x_0, t_0, 1))}. \]
From the assumption $p>n+2$,  we have $(q-1) n/2<1$, which shows
\[ |u \psi | (x,t) \le C A \Vert u \Vert_{L^2(Q^-(x_0, t_0, 1))}. \]
The proposition follows using property (iii).
\end{proof}

\section{Bounds on the heat kernel and its gradient} \label{sec:heatkernelbound}
Now we are ready to prove bounds for the fundamental solution of the heat equation and its gradient. 
A similar lower bound is proven in \cite{Z11:1}. 
An integral average upper bound is proven in \cite{HN:1}, which is very important for us.

\begin{proof}[Proof of Theorem \ref{thkernelUP}]
We may assume by rescaling that $R_0 = 1$.

\textit{Step 1.}
The lower bound follows quickly from Theorem \ref{thdistcon} and (\ref{eq:Prel-lowerhkbound}).
Recall that
\[ K(x,t; y,s) > \frac{C^{-1}}{(t-s)^{n/2}} \exp \Big( { - \frac{d^2_t(x,t)}{2(t-s)}} \Big). \]
Notice that the distance is measured by the metric $g_t$. 
Now we have to get a lower bound involving the distance $d_s (x, y)$, measured by the metric $g_s$. 
By Theorem \ref{thdistcon}, applied multiple times at scale $r_0 \leftarrow A^{-1/2} \sqrt{t-s}$, there is a uniform constant $B = B(A) < \infty$ such that if $d_t (x,y) \geq B \sqrt{t-s}$, then $d_s (x,y) > B^{-1} d_t (x,y)$, which implies
\[
\exp \Big( {- \frac{d^2_t(x, y)}{2(t-s)}  } \Big) \ge \exp \Big( {- \frac{B^{2} d^2_s (x, y)}{2(t-s)} } \Big).
\]
On the other hand, if $d_t (x,y) < B \sqrt{t-s}$, then
\[
\exp \Big( {- \frac{ d^2_t (x, y)}{2 (t-s)} } \Big) \ge \exp \Big({ - \frac{B^2}{2} }\Big)  \ge \exp \Big({ - \frac{B^2}{2} }\Big)  \exp \Big( {- \frac{d^2_s (x, y)}{2(t-s)} }\Big).
\]
So in both cases, we get the desired bound for the right choices of $C_1$ and $C_2$.
Note a similar argument allows us to replace $d_s (x, y)$ by $d_t (x, y)$ in the other heat kernel bounds.

\textit{Step 2.} 
We prove the  bound for $K$.
Recall that by (\ref{eq:Prel-upperhkbound}), there exists $C = C(A) < \infty$ such that
\be
\lab{GdiaJie}
 K(x, t; y, l) \le \frac{C}{(t-l)^{n/2}} \qquad  \text{for} \quad l \in [s, t].
\ee 
Using Theorem \ref{thdistcon} multiple times at scale $r_0 \leftarrow  A^{-1/2} \sqrt{t-s}$, we can find a uniform constant $B = B(A) < \infty$ such that if $d_s (x,y) > B \sqrt{t-s}$, then
\begin{equation} \label{dxyl=dxy0}
B^{-1} d_{l_1} (x,y) \leq d_{l_2} (x,y) \leq B d_{l_1} (x,y) \qquad \text{for all} \quad l_1, l_2 \in [s,t].
\end{equation}
We first consider the case in which $d_s (x,y) \leq  B \sqrt{t-s}$.
By (\ref{GdiaJie}) we have
\[ K(x,t; y, s) \leq \frac{C}{(t-s)^{n/2}}  \exp (B^2 ) \exp \Big( {- \frac{\dist_s^2 (x,y)}{t-s} } \Big). \]
So in this case we are done.
Therefore, assume from now on that (\ref{dxyl=dxy0}) holds.

Next let us recall the following average bound for $K$ in Theorem 1.13 and (1.20) of \cite{HN:1} (cf Theorem 1.30 and (1.32) on their arXiv version), for $l \in [s,t]$:
\begin{multline} \label{eq:HNresult}
\frac{1}{ |B(y,  l, \sqrt{t-l})|_{l}} \int_{B(y, l, \sqrt{t-l})} K (x, t; z, l) dg_l (z) \\
\le \bigg( \int_{B(x,l, \sqrt{t-l})} K(x,t; z, l ) dg_l (z) \bigg)^{-1} \frac{C}{(t-l)^{n/2}} \exp \left( - \frac{d^2_{l} ( x, y)}{8(t-l)}  \right).
\end{multline}
By result of Step 1, we find that there is a uniform constant $c = c(A) > 0$ such that
\[ K(x,t; \cdot, s) > \frac{c}{(t-s)^{n/2}} \qquad \text{on} \quad B(x, l, \sqrt{t-s}). \]
Moreover, by (\ref{eq:kappa1}), we have $|B(x,l, \sqrt{t-l}) |_l > \kappa_1 (t-l)^{-n/2}$.
So the integral on the right-hand side of (\ref{eq:HNresult}) can be bounded from below by a uniform constant and we obtain
\[ \frac{1}{ |B(y,  l, \sqrt{t-l})|_{l}} \int_{B(y, l, \sqrt{t-l})} K (x, t; z, l) dg_l (z) <  \frac{C}{(t-l)^{n/2}} \exp \left( - \frac{d^2_{l} ( x, y)}{8(t-l)}  \right). \]
Using (\ref{eq:kappa2}) and (\ref{dxyl=dxy0}) yields
\[  \int_{B(y, l, \sqrt{t-l})} K (x, t; z, l) dg_l (z) <  C \exp \left( {- \frac{d^2_{s} ( x, y)}{C (t-l)} } \right) \leq C \exp \left( {- \frac{d^2_{s} ( x, y)}{C (t-s)}}  \right). \]
Together with (\ref{GdiaJie}) we get
\[ \int_{B(y, l, \sqrt{t-l})} K^2(x, t; z, l) dg_l (z) < \frac{C}{(t-l)^{n/2}} \exp \left({ - \frac{d^2_s (x, y)}{C(t-s)}}
 \right). \] 
Integrating over the time-interval $[s, (s+ t)/2]$, we see that
\[
\int^{(s+t)/2}_s \int_{B(y, l, \sqrt{t-l})} K^2(x, t; z, l) dg_l (z) dl <
\frac{C (t-s)}{(t-s)^{n/2}} \exp \left( { -  \frac{d^2_s (x, y)}{C(t-s)} } \right).
\] 
Note that $t-l \ge (t-s)/2$ on the time-interval $[s, (s+t)/2]$. 
Hence for $r := A^{-1/2} \sqrt{(t-s)/2}$ we have
\[
 \int_{Q^+ (y, s, r)} K^2(x, t; z, l) dg_l(z) dl <
\frac{C}{(t-s)^{n/2-1}} \exp \left( {- \frac{d^2_s (x, y)}{C(t-s)} } \right).
\]
So by the mean value inequality from Lemma \ref{leCHEmv} we find
\[ K^2 (x,t; y, s) \leq \frac{C}{r^{n+2}} \int_{Q^+ (y, s, r)} K^2(x, t; z, l) dg_l(z) dl < \frac{C}{(t-s)^{n}} \exp \left( {- \frac{d^2_s (x, y)}{C(t-s)} } \right). \]
This implies the upper bound on $K(x,t; y,s)$.

\textit{Step 3.} 
We prove the  bound for $\nabla K$.

Fixing $(y, s)$, we regard $K(x,t) =K(x, t; y, s)$ as a solution to the forward heat equation.
According to (\ref{eq:Prel-derbound}), we have
\[
\frac{| \d_x K(x, t; y, s) |^2}{K^2(x, t; y, s)}
\le \frac{C}{t-s} \log \frac{J}{K(x, t; y, s)}.
\] 
Here $J$ is the maximum of $K(\cdot, \cdot ; y, s)$ on $\M \times [\frac{t+s}{2}, t]$.
Hence
\[
| \d_x K(x, t; y, s) |^2 \le  \frac{C}{t-s} \cdot  J K(x,t;y,s) \cdot \frac{K(x,t;y,s)}{J} \log \frac{J}{K(x,t;y,s)}
\le \frac{C}{t-s} \, J K(x,t; y,s).
\] 
By (\ref{GdiaJie})
\[
J \le \frac{C}{(t-s)^{n/2}}.
\]  
Combining this with the previous inequality, we deduce
\begin{multline*}
| \d_x K(x, t; y, s) |^2 \le  \frac{C}{(t-s)^{1+(n/2)}} K(x,t; y,s) \\
\leq \frac{C}{(t-s)^{1+(n/2)}} \frac{C_1}{(t-s)^{n/2}} \exp \Big( {- \frac{d^2_s (x, y)}{C_2 (t-s)}} \Big).
\end{multline*}
This yields the desired bound for $|\d_x K|$.
\end{proof}

\section{Backward pseudolocality} \label{sec:pseudoloc}
In the following lemma we derive bounds on the Ricci curvature and the time derivative of the curvature tensor that are better than expected, given bounds on the curvature and the scalar curvature on a parabolic neighborhood.
The bound on the Ricci curvature can also be found in \cite[Theorem 3.2]{Wa:1}.
We give a proof for completeness. 

\begin{lemma} \label{Lem:goodRicbound}
There are universal constants $C_0, C_1, \ldots < \infty$ such that the following holds:

Consider a Ricci flow $(\M, (g_t)_{t \in [0,T)})$.
Let $x_0 \in \M$, $t_0 \in [0,T)$ and $0 < r_0^2 <\min \{ 1, t_0 \}$.
Assume that the ball $B(x_0, t_0, r_0)$ is relatively compact and that we have the curvature bounds $|R| \leq n$ and $|{\Rm}| \leq r_0^{-2}$ on the parabolic neighborhood $P = P(x_0,t_0,r_0,-r_0^2)$.

Then $|{\Ric}| (x_0,t_0) < C_0 r_0^{-1}$ and $|\partial_t {\Rm}| (x_0,t_0) < C_0 r_0^{-3}$.
Moreover, we have the bounds $|\nabla^m {\Ric}| (x_0, t_0) < C_0 r_0^{-1-m}$ for all $m \geq 0$.
\end{lemma}

Note that by Shi's estimates (cf \cite{Sh:1}) we only have $|\partial_t {\Rm}| (x,t) < C r^{-4}$.
So the lemma shows that the exponent can be improved under an additional bound on the scalar curvature.

\begin{proof}
We first rescale the given Ricci flow parabolically, along with the radius $r_0$, such that $r_0 = 1$.
Then we have the bounds $|{\Rm}| \leq 1$ and $|R| \leq R_0 := n r_0^2$, for some $R_0 \leq n$, on $P(x_0, t_0, 1, -1)$.
The goal is then to show that $|\nabla^m {\Ric}| (x_0, t_0) < C_m R_0^{1/2}$ and $|\partial_t {\Rm}| (x_0,t_0) < C_0  R_0^{1/2}$.

Observe that by Shi's estimates (cf \cite{Sh:1}) there are universal constants $D_0, D_1, \ldots < \infty$ such that
\begin{equation} \label{eq:Shionr2}
 |{\nabla^m \Rm}| < D_m  \qquad \text{on} \qquad P (x_0, t_0, \tfrac12 , - \tfrac12)
\end{equation}
for all $m = 0,1, \ldots$.
Let $\alpha > 0$ be a small constant, whose value we will determine later, and consider the exponential map
\[ p : B(0,\alpha) \subset \mathbb{R}^n \longrightarrow B(x_0, t_0,\alpha ) \subset \M \]
based at $x_0$ with respect to the metric $g_{t_0}$.
Let $\widetilde{g}_t = p^* g_t$, $t \in [t_0 - \alpha^2, t_0]$ be the rescaled pull back of $g_t$.
Note that in our applications, due to the non-collapsing property (see (\ref{eq:kappa1})), we have a lower bound on the injectivity radius at $x_0$ and hence the map $p$ can be guaranteed to be injective for sufficiently small $\alpha$.
To keep this lemma more general, we will, however, not assume this property and instead proceed as follows:
By Jacobi field comparison and distance distortion estimates under the Ricci flow, we find that there is a universal choice for $\alpha$ such that the metric $\widetilde{g}_t$ is 2-bilipschitz to the Euclidean metric on $\mathbb{R}^n$ for all $t \in [ t_0 - \alpha^2, t_0]$.
Moreover, it follows from (\ref{eq:Shionr2}) that there are universal constants $\widetilde{D_0}, \widetilde{D}_1, \ldots < \infty$ such that
\[ |{\partial^m \widetilde{g}_t}| < \widetilde{D}_m  \qquad \text{on} \qquad B(0, \alpha ) \times [t_0 - \alpha^2 , t_0]. \]

Note that $\widetilde{g}_t$ is still a Ricci flow on $B(0, \alpha r)$, which still satisfies $|{\Rm}| \leq 1$ and $|R| \leq R_0$.
From now on, we will work with the metric $\widetilde{g}_t$ only.
Let now $\phi \in C_0^\infty (B(0, \alpha))$ be a cutoff function that satisfies $0 \leq \phi \leq 1$ everywhere and $\phi = 1$ on $B(0, \tfrac12 \alpha)$.
This cutoff function can be chosen such that $|\partial \phi|, |\partial^2 \phi|$ are bounded by some universal constant $C_1 < \infty$.
This implies that there is a universal constant $C_2 < \infty$ such that
\[ |{\Delta_{\widetilde{g}_t} \phi}| < C_2 \qquad \text{for all} \qquad t \in [t_0 - \alpha^2, t_0]. \]
We now make use of the evolution equation for the scalar curvature of $\widetilde{g}_t$:
\[ \partial_t R =  \Delta_{\widetilde{g}_t} R + 2 |{\Ric}|^2. \]
Integrating this equation against $\phi$ and using integration by parts yields
\begin{multline*}
 \bigg| \partial_t \int_{B(0, \alpha )} R(\cdot, t) \phi  d\widetilde{g}_t - \int_{B(0, \alpha )} 2 |{\Ric}(\cdot, t)|^2 \phi d\widetilde{g}_t \bigg| \\
 \leq C_2 \alpha^{-2}  \int_{B(0,\alpha )} |R(\cdot, t)|  d\widetilde{g}_t + \int_{B(0,\alpha )} |R(\cdot, t)|^2 \phi  d\widetilde{g}_t \\
  < C_3 R_0 + C_3 R_0^2 \leq 2 n C_3 R_0,
\end{multline*}
for some universal $C_3 < \infty$.
Integration in time gives us
\begin{multline*}
\int_{t_0 - \alpha^2}^{t_0} \int_{B(0, \alpha)} |{\Ric}|^2 \phi  d\widetilde{g}_t  dt 
< \int_{B(0, \alpha )} |R(\cdot, t_0)| \phi  d\widetilde{g}(t_0) \\
+ \int_{B(0, \alpha )} |R(\cdot, t_0 - \alpha^2 ) |\phi  d\widetilde{g}(t_0 - \alpha^2 ) + 2 C_3 R_0 \cdot \alpha^2
< C_4 R_0,
\end{multline*}
for some universal $C_4 < \infty$.
Thus, there is a universal $C_5 < \infty$ such that
\[ \Vert {\Ric} \Vert_{L^2 (B(0, \frac12 \alpha ) \times [t_0 - \alpha^2 r_0^2, t_0]))} < C_5 R_0^{1/2}. \]

Note that $\Ric$ satisfies the linear parabolic evolution equation
\begin{equation} \label{eq:Ricevolution}
 (\partial_t - \Delta_{\widetilde{g}_t} - 2 \Rm ) \Ric = 0.
\end{equation}
The coefficients of this equation are universally bounded in every $C^m$-norm.
Hence it follows from standard parabolic theory that for some universal $C_6 < \infty$
\[ |{\Ric}| (x_0, t_0) = |{\Ric}| (0,t_0)  < C_6 
\Vert {\Ric} \Vert_{L^2 (B(0, \frac12 \alpha ) \times [t_0 - \alpha^2, t_0]))} < C_6 C_5 R_0^{1/2}. \]
This establishes the first part of the lemma.

Similarly, or by applying the lemma at smaller scales, we obtain that for some universal constant $C_7 < \infty$
\[ |{\Ric}| < C_7 R_0^{1/2} \qquad \text{on} \qquad B(0,\tfrac14 \alpha) \times [t_0 - \tfrac12 \alpha^2, t_0]. \]
We can now apply the Schauder estimates on (\ref{eq:Ricevolution}) and obtain that for all $m \geq 0$
\[ |\nabla^m {\Ric}|  = |{\nabla^m \Ric}|(0, t_0)  < C'_m R_0^{1/2} \]
for some universal constants $C'_m < \infty$.
For the bound on $|\partial_t \Rm|(x_0, t_0)$ observe that at $(x_0, t_0)$ (compare also with \cite[Lemma 7.2]{Ha:0})
\[ \partial_t \Rm_{abcd} = \Ric_{ai} \Rm_{ibcd} + \Ric_{bi} \Rm_{aicd} 
+ \nabla_{ac} \Ric_{bd} + \nabla_{ad} \Ric_{bc} + \nabla_{bc} \Ric_{ad} - \nabla_{bd} \Ric_{ac}.  \]
So
\[ | {\partial_t \Rm} |(x_0, t_0) \leq C_8 |{\Ric (x_0, t_0)}| \cdot |{\Rm (x_0, t_0)}| + |{\nabla^2 \Ric (x_0,t_0)}| < C_0 R_0^{1/2} \cdot 1 + C'_2 R_0^{1/2}. \]
This finishes the proof of the lemma.
\end{proof} 

With this lemma in hand, we can prove the backwards pseudolocality theorem.

\begin{proof}[Proof of Theorem \ref{Thm:backwpseudoloc}]
First, note that by parabolic rescaling by $r_0^{-2}$ it suffices to prove the following statement:

\begin{claim}
Let $(\M^n, (g_t)_{t \in [0,T)}), 1 < T < \infty$ be a Ricci flow on a compact manifold.
Then there are constants $\varepsilon > 0$, $K < \infty$, which only depend on $\nu[g_0, 2T], n$, such that the following holds:

Let $(x_0, t_0) \in \M \times [1,T)$ and $0 < r_0 \leq 1$ and assume that
\[ |R| \leq n \qquad \text{on} \qquad P(x_0, t_0, r_0, - 2(\varepsilon r_0)^2) \]
and
\[ |{\Rm}(\cdot, t_0)| \leq r_0^{-2} \qquad \text{on} \qquad B(x_0, t_0, r_0). \]
Then
\[ |{\Rm}| < K r_0^{-2} \qquad \text{on} \qquad P(x_0, t_0, \varepsilon r_0, - (\varepsilon r_0)^2). \]
\end{claim}

Note that the theorem follows from the claim if we set $r_0 = 1$.
The lower bound $R \geq -n$ on $P(x_0, t_0, 1, - (2 (\varepsilon r_0)^2) \subset \M \times [t_0 - \frac12, t_0]$ follows as explained in section \ref{sec:Preliminaries} (here we have assumed that $\varepsilon < \frac12$).
Next choose $1 > r_* > 0$, such that
\begin{equation} \label{eq:choiceofr1}
4 C_0 \rho^{-3} <  r_*^{-1}.
\end{equation}
Observe that in order to prove the claim, it suffices to prove the claim for the case in which $r \leq r_*$, as the case $r > r_*$ follows from the case $r = r_*$ by adjusting the constants $\varepsilon, K$ to $\varepsilon r_*$, $K r_*^{-2}$, respectively.
So let us in the following assume that
\[ r_0 \leq r_*. \]

We now choose the constants $\varepsilon, K$ as follows:
\[  \varepsilon := \rho  \qquad \text{and} \qquad K := \varepsilon^{-2} =  \rho^{-2} .  \]

Before we carry out the main proof, we use a point-picking argument to show that, without loss of generality, we can assume that the claim is already true on smaller scales $r_0$ and earlier times $t_0$.
In other words, the proof will utilize ``induction'' on the scale $r_0$:
The claim is obviously true for some very small scale $r_0 > 0$, which may depend on the given Ricci flow, since the curvature on $\M \times [0,t_0]$ is always bounded by some (non-uniform) constant.
In the following proof we will show that if the claim holds at all scales $0 < r \leq \frac{1}2 r_0$, then it also holds for scale $r_0$.
So by induction, the claim is true for all scales $r_0 > 0$.

Let us be more specific now.
We claim that we can impose the following extra assumption:

{\it For any $t \in [t_0 - \frac12 r_0^2, t_0]$, $x \in \M$ and $r \leq \frac12 r_0$ the statement of the claim holds, meaning that}
\begin{multline} \label{eq:extraassumptionpp}
 \textit{If}  \quad |R| \leq n \quad \text{on} \quad P(x,t,r, - 2 (\varepsilon r)^2) \quad \text{and} \quad  |{\Rm (\cdot, t) }| < r^{-2} \quad \textit{on} \quad B(x, t, r), \\
  \quad \textit{then} \quad |{\Rm}| < K r^{-2} \quad \textit{on} \quad P (x,t, \varepsilon r, - (\varepsilon r)^2).
\end{multline}
We will now explain why it suffices to show the claim under this extra assumption.
Assume by contradiction that the claim fails for some choice of $(x_0, t_0, r_0)$, $r_0 \leq r_*$.
So the extra assumption (\ref{eq:extraassumptionpp}) has to fail as well and we can find some $(x, t, r)$ with $t \in [t_0 - \frac12 r_0^2, t_0]$, $r \leq \frac12 r_0$ that violates (\ref{eq:extraassumptionpp}).
In other words, the claim also fails for $(x_0, t_0, r_0) \leftarrow (x_1, t_1, r_1) := (x,t,r)$, which implies that the extra assumption fails as well in this setting.
We can hence choose another triple $(x,t,r)$ with $t \in [t_1 - \frac12 r_1^2, t_1]$ and $r \leq \frac12 r_1$ that violates (\ref{eq:extraassumptionpp}).
Set $(x_2, t_2, r_2) := (x,t,r)$.
Repeating this process, we end up with an infinite sequence $(x_0, t_0, r_0), (x_1, t_1, r_1), \ldots$ such that (\ref{eq:extraassumptionpp}) is violated for $(x, t, r) \leftarrow (x_i, t_i, r_i)$ for all $i = 0, 1, 2, \ldots$.
Moreover $t_{i+1} \in [t_i - \frac12 r_i^2, t_i]$ and $r_{i+1} \leq \frac12 r_i$.
So $t_{i+1} \in [t_0 - r_0^2, t_0]$, $t_{i+1} \geq 2^{-i} t_0$ and $r_i \leq 2^{-i} r_0 \leq \sqrt{t_i}$.
However, since $|{\Rm}|$ is bounded by some (non-universal) constant on $\M \times [0,t_0]$, we obtain that (\ref{eq:extraassumptionpp}) has to be true for some large $i$, giving us the desired contradiction.

We now present the main argument, namely, proving the statement of the claim under the extra assumption  (\ref{eq:extraassumptionpp}).
Invoke Lemma \ref{Lem:secondcutoff} for $x_0 \leftarrow x_0$, $t_0 \leftarrow t_0$, $r_0 \leftarrow  r_0$ and $\tau \leftarrow 2 (\rho r_0)^2$ to obtain the cutoff function $\phi \in C^0 ( \M \times [ t_0 - 2 \varepsilon^2 r_0^2, t_0 ])$ satisfying the desired properties (a)--(e) of that lemma.

By property (a)
\[ |{\Rm} (\cdot, t_0) |^{-1/2} \geq r_0 > \phi (\cdot, t_0). \]
Choose $\overline{t} \in [t_0 - \varepsilon^2 r_0^2, t_0)$ minimal with the property that
\begin{equation} \label{eq:defoftbar}
 |{\Rm}|^{-1/2} \geq \phi \qquad \text{on} \qquad  \M \times [\overline{t}, t_0].
\end{equation}
We will show that $\overline{t} = t_0 -  \varepsilon^2 r_0^2$.
Assume that, on the contrary, $\overline{t} > t_0 -  \varepsilon^2 r_0^2$.
Then there is a point $\overline{x} \in \M$ such that
\begin{equation} \label{eq:defofxbar}
 |{\Rm(\overline{x}, \overline{t})}|^{-1/2} = \phi (\overline{x}, \overline{t}).
\end{equation}
Set
\begin{equation} \label{eq:defofovr}
 \overline{r} := \tfrac12 |{\Rm} (\overline{x}, \overline{t}) |^{-1/2} = \tfrac12 \phi (\overline{x}, \overline{t}) < \tfrac12 r_0.
\end{equation}
Then, since $\phi (\cdot, \overline{t})$ is $1$-Lipschitz with respect to the metric $g_{\overline{t}}$, we have
\begin{equation} \label{eq:subcurvbound}
 |{\Rm(\cdot, \overline{t})}|^{-1/2} \geq \phi (\cdot, \overline{t}) > \tfrac12 \phi (\overline{x}, \overline{t}) = \overline{r} \qquad \text{on} \qquad B(\overline{x}, \overline{t}, \overline{r}).
\end{equation}
So by Lemma \ref{Lem:secondcutoff}(c) we have $B(\overline{x}, \overline{t}, \overline{r}) \subset B(x_0, t_0, r_0)$, which implies that
\begin{equation} \label{eq:subscalbound}
 |R| \leq n  \qquad \text{on} \qquad P(\overline{x}, \overline{t}, \overline{r}, - 2 (\varepsilon \overline{r})^2).
\end{equation}

The bounds (\ref{eq:subcurvbound}) and (\ref{eq:subscalbound}) give us the right to apply the extra assumption (\ref{eq:extraassumptionpp}) for $x \leftarrow \overline{x}$, $t \leftarrow \overline{t}$ and $r \leftarrow \overline{r}$ to conclude that
\[ |{\Rm}| < K \overline{r}^{-2} = \varepsilon^{-2} \overline{r}^{-2} \qquad \text{on} \qquad P(\overline{x}, \overline{t}, \varepsilon \overline{r}, - (\varepsilon \overline{r})^2). \]
Then by Lemma \ref{Lem:goodRicbound} for $r_0 \leftarrow \varepsilon \overline{r}$ we have
\[ |\partial_t | {\Rm (\overline{x}, \overline{t})}| | < C_0 \varepsilon^{-3} \overline{r}^{-3}. \]
And thus by the definition of $\overline{r}$ (see equation (\ref{eq:defofovr}))
\[ |\partial_t |{\Rm (\overline{x}, \overline{t})}|^{-1/2} | < \tfrac12 C_0 \varepsilon^{-3} \overline{r}^{-3} |{\Rm (\overline{x}, \overline{t} )}|^{-3/2} = 4 C_0 \varepsilon^{-3} .  \]
Using (\ref{eq:defoftbar}), (\ref{eq:defofxbar}) and property (e) of $\phi$ (compare with Lemma \ref{Lem:secondcutoff}) we conclude however
\[ \partial_t |{\Rm}|^{-1/2} (\overline{x}, \overline{t} ) \geq \partial_t \phi (\overline{x}, \overline{t}) >  r_0^{-1} \geq r_*^{-1}. \]
So
\[ 4 C_0 \varepsilon^{-3} >  r_*^{-1} . \]
This inequality contradicts our choice of $r_*$ (see inequality (\ref{eq:choiceofr1})), hence showing that our assumption $\overline{t} > t_0 -  \varepsilon^2 r_0^2$ was wrong.

So $\overline{t} = t_0 -  \varepsilon^2 r_0^2$, which implies
\[ |{\Rm}|^{-1/2} \geq \phi \qquad \text{on} \qquad \M \times [t_0 - \varepsilon^2 r_0^2, t_0]. \]
Thus we can use property (b) of $\phi$ to conclude that
\[ |{\Rm}| \leq \phi^{-2} < \rho^{-2}  r_0^{-2} = K r_0^{-2}  \qquad \text{on} \qquad P(x_0, t_0, \rho  r_0, -  \varepsilon^2 r_0^2). \]
Since $\varepsilon = \rho$, this finishes the proof.
\end{proof}

Combining the backward pseudolocality theorem, Theorem \ref{Thm:backwpseudoloc}, with Perelman's forward pseudolocality Theorem gives us an improved pseudolocality result in both time directions.

\begin{proof}[Proof of Corollary \ref{costrongreg}]
By parabolic rescaling we may assume that $r_0 = 1$.
This implies, as explained in section \ref{sec:Preliminaries}, that $t_0 \geq 1$, $T - t_0 \geq 1$ and $|R| \leq n$ on $B(x_0, t_0, r_0) \times [t_0 - 2 \varepsilon^2, t_0 + \varepsilon^2]$.

Let $\delta > 0$ be the constant from Perelman's pseudolocality theorem \cite[Theorem 10.1]{P:1} for $\alpha \leftarrow 1$ and let $\eps > 0$ be less than the minimum of the corresponding constants in the backwards pseudolocality Theorem \ref{Thm:backwpseudoloc} and Perelman's pseudolocality theorem, again for $\alpha \leftarrow 1$.
Moreover, by Theorem \ref{thdistcon}, we may choose $\varepsilon$ small enough such that
\[ B(x_0, t, \varepsilon) \subset B(x_0, t_0, 1) \qquad \text{for all} \qquad t \in [t_0, t_0 + \varepsilon^2]. \]
By Perelman's pseudolocality theorem, we have for all $t \in (t_0, t_0 + \eps^2]$ and $x \in B( x_0, t, \eps )$
\[ |{\Rm}|(x,t) \leq (t-t_0)^{-1} + \eps^{-2}. \]
For any $t \in (t_0, t_0 +  \eps^2]$ and $x \in B(x_0, t, \frac12 \eps )$ we can now apply Theorem \ref{Thm:backwpseudoloc} with $r_0 \leftarrow \frac12 (t - t_0)^{1/2} \leq \frac12 \eps$ and obtain that
\begin{equation} \label{eq:Rmboundonsmallball}
 |{\Rm}| < K \big( \tfrac12 (t - t_0)^{1/2} \big)^{-2} \qquad \text{on} \qquad 
 P\big( x,t, \tfrac12 \eps (t - t_0)^{1/2}, - \big(\tfrac12 \eps  (t - t_0)^{1/2} \big)^2 \big).
\end{equation}
Here $K$ is the constant from Theorem \ref{Thm:backwpseudoloc}.
Next, using Lemma \ref{Lem:goodRicbound}, we can find a constant $K_* < \infty$, which only depends on $\nu[g_0, 2T], n$, such that:
\begin{multline} \label{eq:RicanddtRm}
|{\Rm}| (x,t) < K_* (t-t_0)^{-1}, \qquad |{\Ric}|(x,t) < K_* (t-t_0)^{-1/2} \qquad \\ \text{and} \qquad |\partial_t {\Rm} |(x,t) < K_* (t-t_0)^{-3/2}.
\end{multline}
The first inequality is just a restatement of (\ref{eq:Rmboundonsmallball}).

Choose $\eps_* := \min \{ \eps, (16 K_*)^{-1} \eps \}$.
We claim that for all $t \in [t_0, t_0 + \eps_*^2 ]$, we have
\begin{equation} \label{eq:ballintwiceball}
B(x_0, t, \tfrac18 \eps ) \subset B(x_0, t_0, \tfrac14 \eps ) \subset B(x_0, t, \tfrac12 \eps ).
\end{equation}
Assume that the second inclusion does not hold for all $t \in [t_0, t_0 + \eps_*^2 ]$.
Choose $\overline{t} \in [t_0, t_0 + \eps_*^2 )$ maximal such that it holds for all $t \in [t_0, \overline{t}]$.
Then there is a point $\overline{x} \in \overline{B(x_0, t_0, \frac14 \eps  )}$ such that $d_{\overline{t}}(\overline{x}, x_0) = \frac12 \eps $.
For any $t \in [t_0, \overline{t}]$, the point $\overline{x}$ is contained in the closure of $B(x_, t, \frac12 \eps )$ and thus, by the discussion in the previous paragraph, at every time $t \in [t_0, \overline{t}]$, all points on a minimizing geodesic between $\overline{x}$ and $x_0$ satisfy (\ref{eq:RicanddtRm}).
It follows that (assuming $\eps < 1$)
\begin{equation} \label{eq:distdistortiont-12}
  |{\partial_t d_t (\overline{x}, x_0)}| <  K_* (t - t_0)^{-1/2} d_t ( \overline{x}, x_0 ) < K_* (t - t_0)^{-1/2}.
\end{equation}
Integrating this inequality from $t_0$ to $\overline{t}$ yields a contradiction:
\[ \tfrac12 \eps  = d_{\overline{t}} (\overline{x}, x_0)  < d_{t_0} (\overline{x}, x_0) + 2 K_* (\overline{t} - t_0)^{1/2} < \tfrac14 \eps  + 2 K_* \eps_*  < \tfrac12 \eps . \]
In order to show the first inclusion of (\ref{eq:ballintwiceball}), we choose $\overline{t} \in [t_0, t_0 + \eps_*^2]$ maximal such that (\ref{eq:ballintwiceball}) holds for all $t \in [t_0, \overline{t}]$.
If $\overline{t} < t_0 + \eps_*^2$, then we can find a point $\overline{x} \in \overline{B(x_0, \overline{t}, \frac18 \eps)}$ with $d_{t_0} (\overline{x}, x_0) = \frac14 \eps$.
Integrating (\ref{eq:distdistortiont-12}) yields again a contradiction:
\[ \tfrac14 \eps = d_{t_0} (\overline{x}, x_0) < d_{\overline{t}} (\overline{x}, x_0) + 2 K_* (\overline{t} - t_0)^{1/2} < \tfrac18 \eps + 2 K_* \eps_*  < \tfrac14 \eps. \]

In view of the discussion in the beginning of the proof, the second inclusion in (\ref{eq:ballintwiceball}) implies that the curvature and curvature derivative estimates in (\ref{eq:RicanddtRm}) hold on the parabolic neighborhood $P(x_0, t_0, \frac14 \eps, \eps_*^2)$.
This fact allows us to integrate the bound on $|\partial_t {\Rm}|$ from $t_0 + \eps_*^2$ to any $t \in [t_0, t_0 + \eps_*^2]$ and obtain that for all $x \in B(x_0, t_0, \tfrac14 \eps)$
\begin{multline*}
 |{\Rm}| (x,t) \leq |{\Rm}| (x, t_0 + \eps_*^2) + \int_t^{t_0 + \eps_*^2} |\partial_t {\Rm}| (x, t') dt' \\
 < K_* \eps_*^{-2}  + \int_t^{t_0 + \eps_*^2}K_* (t'-t_0)^{-3/2} dt'
 < K_* \eps_*^{-2}  + 2 K_* (t - t_0)^{-1/2} \\
 < 10 K_* \eps_*^{-1}  (t-t_0)^{-1/2}.
\end{multline*}
So by the first inclusion in (\ref{eq:ballintwiceball}) we find that for all $t \in [t_0, t_0 + \eps_*^2]$ we have
\[ |{\Rm}| (\cdot, t) < 10 K_* \eps_*^{-1}  (t-t_0)^{-1/2} \qquad \text{on} \qquad B(x_0, t, \tfrac18 \eps ). \]

We now apply the backward pseudolocality theorem, Theorem \ref{Thm:backwpseudoloc}, once again.
Let us first choose $\beta := (10 K_* \eps_*^{-1} )^{-1/2}$ and $\theta := (\frac18 \eps)^4$.
Then for any $t \in [t_0, t_0 + \theta ]$, we have $\beta  (t-t_0)^{1/4} \leq \frac18 \eps$.
So for any $t \in [t_0, t_0 + \theta ]$
\[ |{\Rm}| (\cdot, t) < \big( \beta (t-t_0)^{1/4} \big)^{-2} \qquad \text{on} \qquad B \big( x_0, t, \beta  (t-t_0)^{1/4} \big). \]
Theorem \ref{Thm:backwpseudoloc} then yields
\begin{equation} \label{eq:boundatsomet}
 |{\Rm}| < K \beta^{-2}  (t-t_0)^{-1/2} \qquad \text{on} \qquad P \big( x_0, t, \eps \beta  (t-t_0)^{1/4}, - \eps^2 \beta^2  (t-t_0)^{1/2} \big).
\end{equation}
We now choose $0 < \theta_0 < \theta$ uniformly such that
\[  \eps^2 \beta^2 \theta_0^{1/2} > 2 \theta_0. \]
Applying (\ref{eq:boundatsomet}) with $t = t_0 + \theta_0$ yields
\[ |{\Rm}| < K \beta^{-2} \theta_0^{-1/2}  \qquad \text{on} \qquad P(x_0, t_0 + \theta_0 , \theta_0^{1/2}, - 2 \theta_0 ). \]
So we obtain a uniform curvature bound in a parabolic neighborhood that intersects the time-slice $t$.
Due to this curvature bound, distances within this parabolic neighborhood only change up to a fixed factor, which implies that the parabolic neighborhood contains a ball $B(x_0, t_0, \varepsilon')$ at time $t_0$ for some uniform $\varepsilon' > 0$.
This finishes the proof of the corollary. 
\end{proof}

\section{$L^2$ curvature bound in dimension 4} \label{sec:L2bound}
In this section we consider Ricci flows in dimension $4$.
We use the backwards pseudolocality theorem to derive an $L^2$ bound on the Riemannian curvature and an $L^p$ bound on the Ricci curvature for $0 < p < 4$.
Using these bounds, we eventually establish Corollary \ref{Cor:convergenceorbifold}.

We will need the following lemma.

\begin{lemma} \label{Lem:picay}
Consider a Ricci flow $(\M^4, (g_t)_{t \in [0,T)})$ on a compact $4$-dimensional manifold.
Suppose that $|R| \leq 1$ everywhere.
Then there is a constant $\delta > 0$, which only depends on $\nu[g_0, 2T], n$, such that the following holds:

Let $(x, t) \in \M \times [0,T)$ and assume that $\rrm^2 (x, t) \leq t$ (recall that $\rrm$ was defined in Definition \ref{Def:rho}).
Then there is a $y \in \M$ such that
\begin{enumerate}[label=(\alph*)]
\item $d_t (x,y) < 2 \rrm (x, t)$,
\item $\rrm (y, t) \leq \rrm (x, t)$,
\item $\int_{B(y, t, \frac16 \rrm (y,t))} |{\Rm}|^2 dg_t > \delta$.
\end{enumerate}
\end{lemma}

\begin{proof}
We will choose $y$ by a point-picking process.
Construct a sequence $y_1, y_2, \ldots \in \M$ by the following inductive process:
Let $y_0 := x$.
If $y_k$, $k \geq 0$, is already chosen, then by the definition of $\rrm$ there is a point $y' \in \overline{B} (y_k, t, \rrm (y_k, t))$ for which $|{\Rm}| (y', t) = \rrm^{-2} (y_k, t)$.
Set $y_{k+1} := y'$.
Summarizing, we have
\[ d_t (y_{k+1}, y_k) \leq \rrm (y_k, t) \qquad \text{and} \qquad |{\Rm}| (y_{k+1}, t) = \rrm^{-2} (y_k, t) \]
for all $k = 0,1, \ldots$.
So
\[ \rrm^{-2} (y_k, t) = |{\Rm}| (y_{k+1}, t)  \leq \rrm^{-2} (y_{k+1}, t), \]
which implies that
\begin{equation} \label{eq:smallerandsmaller}
 \rrm (x,t) = \rrm (y_0, t) \geq \rrm (y_1, t) \geq \rrm (y_2, t) \geq \ldots . 
\end{equation}

Next, observe that the sequence $\rrm (y_k, t)$ is bounded from below, because $|{\Rm}| (\cdot, t)$ is bounded on $\M$.
This implies that for some $k = 0, 1, \ldots$ we have
\begin{equation} \label{eq:rrmnotlessthan10}
 \rrm (y_{k+1}, t) > \tfrac1{10} \rrm (y_k, t),
\end{equation}
because otherwise the sequence $\rrm (y_k, t)$ would exponentially converge to $0$.
Amongst those $k = 0,1, \ldots$ that satisfy (\ref{eq:rrmnotlessthan10}) choose the smallest $k$ with this property, i.e.
\[ \rrm (y_1, t) \leq \tfrac1{10} \rrm (y_0, t), \qquad \ldots, \qquad \rrm (y_k, t) \leq \tfrac1{10} \rrm (y_{k-1}, t) . \]
We now set $y := y_{k+1}$.

We will show that $y$ satisfies assertions (a)--(c).
Assertion (b) holds due to (\ref{eq:smallerandsmaller}).
For assertion (a) we estimate
\begin{multline*}
d_t (x, y_{k+1} ) \leq d_t (y_0, y_1)  + \ldots + d_t (y_k, y_{k+1}) \leq \rrm (y_0, t) + \dots + \rrm (y_k, t) \\
< \big(1 + \tfrac1{10} + \ldots + \tfrac1{10^k} \big) \rrm (x, t) < 2 \rrm (x,t).
\end{multline*}

It remains to establish assertion (c).
For this observe first that by the inductive choice of $y_{k+1}$ and (\ref{eq:rrmnotlessthan10}), we have
\begin{equation} \label{eq:yk1large}
 |{\Rm}| (y_{k+1}, t) = \rrm^{-2} (y_k, t) > \frac1{100} \rrm^{-2} (y_{k+1}, t).
\end{equation}
If we apply the backward pseudolocality theorem, Theorem \ref{Thm:backwpseudoloc}, at $y = y_{k+1}$, then we obtain that
\[ |{\Rm}| < K \rrm^{-2} (y,t) \qquad \text{on} \qquad P(y,t, \varepsilon \rrm (y,t), - \varepsilon^2 \rrm^2 (y,t)). \]
Hence by Shi's estimates there is a constant $C < \infty$ such that
\[  |\nabla {\Rm} (\cdot, t)| < C \rrm^{-3} (y, t) \qquad \text{on} \qquad B(y, t, \tfrac12 \varepsilon \rrm (y,t)). \]
Integrating this bound and using (\ref{eq:yk1large}) yields with $\beta := \min \{ \frac1{200} C^{-1}, \tfrac12 \varepsilon \}$.
\[ |{\Rm} (\cdot, t) | > \frac1{200} \rrm^{-2} (y,t) \qquad \text{on} \qquad B(y, t, \beta \rrm (y,t)). \]
It follows that
\begin{multline*}
 \int_{B(y, t, \frac16 \rrm (y,t))} |{\Rm}|^2 (\cdot, t) dg_t \geq \int_{B(y, t, \beta \rrm (y, t))} |{\Rm}|^2 (\cdot, t) dg_t \\ > \Big( \frac1{200} \rrm^{-2} (y,t) \Big)^2 \cdot \kappa_1 \big( \beta \rrm (y,t) \big)^4 = \frac{\kappa_1 \beta^4}{200^2} =: \delta > 0. 
\end{multline*}
So assertion (c) holds.
This proves the desired result.
\end{proof}

\begin{proof}[Proof of Theorem \ref{Thm:L2bound}.]
First note that by parabolic rescaling we may assume that $R_0 = 1$.

Let $\overline{r} > 0$ be a parameter whose value we will fix in the course of the proof, depending only on $T$ and $\nu [ g_0, 2T ]$.
For the rest of this proof fix some $t \in [T/2, T)$.

Let $\delta > 0$ be the constant from Lemma \ref{Lem:picay} and set
\[ S := \Big\{ y \in \M \;\; : \;\; \rrm (y,t) < \overline{r} \qquad \text{and} \qquad \int_{B(y, t, \frac16 \rrm (y,t))} |{\Rm}|^2 > \delta \Big\}. \]
Now consider all balls $B(y, t, \frac16 \rrm (y,t))$ for $y \in S$.
By Vitali's Covering Theorem there are finitely many points $y_1, \ldots, y_N \in S$ such that for
\[ r_i := \rrm (y_i, t) < \overline{r} \]
the balls $B(y, t, \tfrac16 r_i)$ are pairwise disjoint and
\begin{equation} \label{eq:Vitali}
 \bigcup_{i=1}^N B(y_i, t, \tfrac12 r_i) \supset \bigcup_{y \in S} B(y, t, \tfrac16 \rrm (y,t)) \supset S.
\end{equation}

\begin{claim}
For every $x \in \M$ with $\rrm (x,t) < \overline{r}$ there is an $i \in \{ 1, \ldots, N \}$ for which
\[ d_t (x, y_i) < 2 \rrm (x,t) + \tfrac12 r_i. \]
\end{claim}

\begin{proof}
Let $x \in \M$ with $\rrm (x,t) < \overline{r}$.
By Lemma \ref{Lem:picay} there is a $y \in \M$ with $d_t (x,y) < 2 \rrm(x,t)$ and $\rrm (y,t) \leq \rrm (x,t) < \overline{r}$ that satisfies the lower bound on the integral in the definition of $S$.
So altogether $y \in S$.
By (\ref{eq:Vitali}) there is an index $i \in \{ 1, \ldots, N \}$ for which $y \in B(y_i, t, \tfrac12 r_i)$.
So we conclude
\[ d_t (x, y_i) \leq d_t (x, y) + d_t( y, y_i) < 2 \rrm (x,t) + \tfrac12 r_i. \]
This finishes the proof of the claim.
\end{proof}

Next observe that for any $p > 0$, by Fubini's Theorem
\begin{multline} \label{eq:coareaformu}
 \int_{\M} \rrm^{-p} (x, t) dg_t(x) = \int_{\M} \bigg( \overline{r}^{-p} + \int_{\rrm (x, t)}^{\overline{r}} p s^{-p-1}   ds  \bigg) dg_t(x) \\
 \leq \overline{r}^{-p} \vol_t \M + \int_0^{\overline{r}} p s^{-p-1} | \{ \rrm (\cdot, t) \leq s \} |_t  ds.
\end{multline}

We will now split the sub level sets inside the integral on the right-hand side into $N$ different domains, depending on which of the expressions $d_t( \cdot, y_i) - \frac12 r_i$ is minimal.
To do this set for $i = 1, \ldots, N$
\begin{multline*}
 D_i := \{ x \in \M \;\; : \;\; \rrm (x,t) \leq \overline{r} \qquad  \text{and} \\ \qquad d_t(x, y_i) - \tfrac12 r_i \leq d_t (x, y_j) - \tfrac12 r_j \qquad \text{for all} \qquad j = 1, \ldots, N \}.
\end{multline*}
Then
\begin{equation} \label{eq:DicoverB}
  \{ \rrm (\cdot, t) \leq \overline{r} \} = D_1 \cup \ldots \cup D_N. 
\end{equation}
It follows from the Claim that for all $x \in D_i$
\[ \rrm (x,t) > \tfrac12 ( d_t (x, y_i) - \tfrac12 r_i ) . \]
So if $d_t (x, y_i) \geq \tfrac34 r_i$, then $\rrm (x,t) > \frac12 (d_t (x, y_i) - \frac12 \cdot \frac43 d_t (x, y_i)) > \tfrac16 d_t (x, y_i)$.
On the other hand, by the definition of $\rrm$, we obtain that if $d_t (x, y_i ) < \tfrac34 \rrm (y_i, t) = \tfrac34 r_i$, then $\rrm (x, t) > \tfrac14 r_i > \frac13 d_t (x, y_i) > \frac16 d_t (x, y_i)$.
Hence, in either case
\begin{equation} \label{eq:rrm>16d}
 \rrm (x, t) > \tfrac16 d_t (x, y_i) \qquad \text{if} \qquad x \in D_i. 
\end{equation}
Using this inequality, we can now compute that for any $s \leq \overline{r}$
\begin{multline} \label{eq:areabound}
 | (\{ \rrm (\cdot, t) \leq s \} |_t \leq \sum_{i = 1}^N | \{ \rrm (\cdot, t) \leq s \} \cap D_i |_t \\
 \leq \sum_{i = 1}^N | \{ d_t (\cdot, y_i) \leq 6 s \} \cap D_i |_t \leq \sum_{i=1}^N | B( y_i, t, 6s) |_t \leq (6^4 \cdot \kappa_2) \cdot N s^4.
\end{multline}
So using (\ref{eq:coareaformu}), we obtain for $0 < p < 4$
\begin{multline}
\label{eq:NdependentLponrrm}  
 \int_M \rrm^{-p} (x, t)  dg_t(x) \leq \overline{r}^{-p} \vol_t M + \int_0^{\overline r} p s^{-p - 1} | \{ \rrm (\cdot, t) < s \} |_t  ds \\
 \leq \overline{r}^{-p} \vol_t M  + C N p \int_0^{\overline{r}} s^{3-p} ds \leq \overline{r}^{-p} \vol_t M + CN \frac{p}{4-p} \overline{r}^{4-p}.
\end{multline}
Here $C < \infty$ is a constant, which only depends on $\kappa_2$.

Next, assume that $\overline{r} < \min \{ 1, \sqrt{T/2} \}$ and observe that by Theorem \ref{Thm:backwpseudoloc} and Lemma \ref{Lem:goodRicbound} there is a universal constant $C' < \infty$ such that for all $x \in \M$ for which $\rrm (x,t) \leq \overline{r}$, we have
\[ |{\Ric}|(x,t) \leq C' \rrm^{-1} (x,t). \]
So by (\ref{eq:NdependentLponrrm}) for $p = 2$, we get
\begin{equation} \label{eq:RicbounddependingonN}
 \int_{\M} |{\Ric} (x, t) |^2  dg_t(x) \leq C^{\prime 2} \int_{\M} \rrm^{-2} (x, t) dg_t(x) \leq C^{\prime 2} \overline{r}^{-2} \vol_t M + C C^{\prime 2} N \overline{r}^{2}.
\end{equation}

On the other hand, recall that by the choice of the $y_i \in S$ we have that for all $i = 1, \ldots, N$
\[ \int_{B(y_i, t, \frac16 r_i)} |{\Rm}( x,t)|^2  dg_t(x) > \delta \]
and the domains of these integrals are pairwise disjoint.
So
\begin{equation} \label{eq:RmbounddependingonN}
 \int_{\M} |{\Rm}(x, t)|^{2} dg_t(x) > N \delta.
\end{equation}

Finally, we apply the Theorem of Chern-Gau\ss-Bonnet:
\begin{equation} \label{eq:ChernGaussBonnet}
 \int_{\M} |{\Rm}(x,t)|^2  dg_t(x) = 32 \pi^2 \chi(M) + \int_{\M} \big( 4 |{\Ric} (x,t)|^2 - R^2 (x,t) \big) dg_t(x).
\end{equation}
Combining this identity with (\ref{eq:RicbounddependingonN}) and (\ref{eq:RmbounddependingonN}) gives us
\[ N \delta < 32 \pi^2 \chi(M) + 4 C^{\prime 2} \overline{r}^{-2} \vol_t M + 4 C C^{\prime 2} N \overline{r}^2. \]
We now choose $\overline{r} > 0$ small enough such that
\[ 4 C C^{\prime 2} \overline{r}^2 < \tfrac12 \delta. \]
Then
\begin{equation} \label{eq:Nboundfinal}
  N < \frac{32 \pi^2 \chi(M) + 4 C^{\prime 2} \overline{r}^{-2} \vol_t M}{\frac12 \delta} = A \chi(M) + B \vol_t M.
\end{equation}

The theorem now follows using the previous inequalities:
The $L^p$-bound on $\Ric$ follows from (\ref{eq:NdependentLponrrm}) together with (\ref{eq:Nboundfinal}).
The $L^2$-bound on $\Rm$ is a consequence of this bound for $p=2$ and (\ref{eq:ChernGaussBonnet}).
And the last statement follows from (\ref{eq:areabound}) and (\ref{eq:Nboundfinal}).
Note that the case $\overline{r} \leq s \leq 1$ follows by adjusting $A$ and $B$.
\end{proof}

The remainder of this section is devoted to the proof of Corollary \ref{Cor:convergenceorbifold}.
For the following we fix a Ricci flow $(\M^4, (g_t)_{t \in [0,T)})$, $T < \infty$ on a $4$-dimensional manifold and assume that $R \leq R_0 < \infty$ everywhere.
Let us first summarize all the estimates that we know so far.
By (\ref{eq:Prel-volelementdistortion}) we have $C_1^{-1} dg_{t_1} < dg_{t_2} < C_1 dg_{t_1}$ for any $t_1, t_2 \in [0, T)$ and by (\ref{eq:Prel-voldistortion})
\[ C_2^{-1} < \vol_t \M < C_2. \]
We also have the non-collapsing and non-inflating properties (\ref{eq:kappa1}), (\ref{eq:kappa2}) for some uniform $\kappa_1, \kappa_2 > 0$.
These properties, combined with the upper and lower volume bound, imply a uniform upper and lower diameter bound.
\[ C_3^{-1} < \diam_t \M < C_3. \]
We will also use the following two consequences of Theorem \ref{Thm:L2bound}:
\begin{equation} \label{eq:RmL2}
 \int_{\M} |{\Rm}(\cdot, t_k)|^2 < C_4
\end{equation}
and for any $0 < s \leq 1$
\begin{equation} \label{eq:rrmvolboundC2}
\frac{| \{ \rrm(\cdot, t) \leq s \}|_{t}}{s^4} < C_5.
\end{equation}

Finally, we mention a consequence of the proof of Theorem \ref{Thm:L2bound}.

\begin{lemma} \label{Lem:yit}
Let $(\M^4, (g_t)_{t \in [0,T)})$, $T < \infty$ be a Ricci flow on a $4$-dimensional manifold and assume that $R \leq R_0 < \infty$ everywhere.
Then there is a natural number $N \in \mathbb{N}$ such that for every time $t \in [0,T)$, we can find points $y_{1,t}, \ldots, y_{N,t} \in \M$ such that for any $x \in \M$ we have
\begin{equation} \label{eq:rrm>16dagain}
 \rrm (x,t) > \tfrac16 \min_{i = 1, \ldots, N} d_t (x, y_{i,t}).
\end{equation}
\end{lemma}

\begin{proof}
The lemma is a consequence of the proof of Theorem \ref{Thm:L2bound}.
Consider the points $y_1, \ldots, y_N$ and the radii $r_i = \rrm (y_i, t)$ as constructed in this proof before equation (\ref{eq:Vitali}) for $\overline{r} = \max_{\M} \rrm (\cdot, t)$.
Then by (\ref{eq:DicoverB}) we have $D_1 \cup \ldots \cup D_N = \M$.
So (\ref{eq:rrm>16d}) implies (\ref{eq:rrm>16dagain}).

By (\ref{eq:RmbounddependingonN}) and (\ref{eq:RmL2}) we get that $N$ is bounded by a constant, which is uniform in time.
\end{proof}

The following definition will help us understand the formation of singularities.

\begin{definition}
Let $(\M, (g_t)_{t \in [0,T)})$, $T < \infty$ be a Ricci flow.
A pointed, complete metric space $(X, d, x_\infty)$, $x_\infty \in X$ is called a \emph{limit} of $(\M, (g_t)_{t \in [0,T)})$ if there is a sequence of times $t_k \nearrow T$, a sequence of points $x_k \in \M$ and a sequence of numbers $\lambda_k \geq 1$ such that the length metrics of $(\M, \lambda_k^2 g_{t_k}, x_k)$ converge to $(X, d, x_\infty)$ in the Gromov-Hausdorff sense.
If $\lim_{k \to \infty} \lambda_k = \infty$, then $(X, d, x_\infty)$ is called a \emph{blowup limit}.
\end{definition}

Due to the non-collapsing and non-inflating properties, (\ref{eq:kappa1}), (\ref{eq:kappa2}), (blowup) limits always have to exist:

\begin{lemma}
Let $(\M^4, (g_t)_{t \in [0,T)})$, $T < \infty$ be a Ricci flow on a $4$-dimensional manifold and assume that $R \leq R_0 < \infty$ everywhere.
Consider arbitrary sequences $t_k \nearrow T$, $x_k \in \M$ and $\lambda_k \geq 1$.
Then, after passing to a subsequence, $(\M, \lambda_k^2 g_{t_k}, x_k)$ converges to a complete limit $(X, d, x_\infty)$ in the Gromov-Hausdorff sense.
\end{lemma}

We will now analyze (blowup) limits more carefully.
The next lemma states that such limits are Riemannian manifolds away from finitely many points.

\begin{lemma} \label{Lem:Xsmoothonreg}
Let $(\M^4, (g_t)_{t \in [0,T)})$, $T < \infty$ be a Ricci flow on a $4$-dimensional manifold and assume that $R \leq R_0 < \infty$ everywhere.
Let $(X, d, x_\infty)$ be a limit of $(\M^4, (g_t)_{t \in [0,T)})$.
Then $\diam (X, d) > 0$ and there are points $y_{1,\infty}, \ldots, y_{N,\infty} \in X$ such that the metric $d$ restricted to $X \setminus \{ y_{1,\infty}, \ldots, y_{N,\infty} \}$ is induced by a smooth Riemannian metric $g_\infty$.
Moreover, for all $x \in X$
\[ |{\Rm}_{g_\infty} |(x) \leq 36 \max_{i = 1, \ldots, N} d^{-2} (x, y_{i,\infty}) \]
and
\begin{equation} \label{eq:RmXL2}
 \Vert {\Rm}_{g_\infty} \Vert_{L^2(X \setminus \{ y_{1,\infty}, \ldots, y_{N,\infty} \})} \leq C_4.
\end{equation}
If $(X, d, x_\infty)$ is even a blowup limit, then $\Ric_{g_\infty} \equiv 0$.
\end{lemma}

\begin{proof}
Consider the points $y_{1,t_k}, \ldots, y_{N,t_k}$ from Lemma \ref{Lem:yit}.
After passing to a subsequence, we may assume that these points converge to points $y_{1, \infty}, \ldots, y_{N, \infty} \in X$.
By the backward pseudolocality Theorem \ref{Thm:backwpseudoloc} and Shi's estimates, we have uniform bounds on the covariant derivatives of the curvature tensor at uniform distance away from the $y_{i, t_k}$.
So $(\M, g_{t_k})$ smoothly converges to a Riemannian metric $g_\infty$ on $X \setminus \{ y_{1, \infty}, \ldots, y_{N, \infty} \}$.
The curvature bounds descend to the limit.
The vanishing of $\Ric_{g_\infty}$ in the blowup case is a consequence of Lemma \ref{Lem:goodRicbound}.
\end{proof}

Next, we prove that (blowup) limits look like orbifolds.

\begin{lemma} \label{Lem:itsanorbifold}
Let $(\M^4, (g_t)_{t \in [0,T)})$, $T < \infty$ be a Ricci flow on a $4$-dimensional manifold and assume that $R \leq R_0 < \infty$ everywhere.
Let $(X, d, x_\infty)$ be a limit of $(\M^4, (g_t)_{t \in [0,T)})$.
Then the tangent cone at every point of $(X,d)$ is isometric to a finite quotient of Euclidean space $\IR^4$.
So $(X, d)$ is diffeomorphic to an orbifold with cone singularities.
\end{lemma}

\begin{proof}
It follows from (\ref{eq:RmXL2}) that the tangent cones around any $y_{i, \infty}$ are flat away from the tip and unique.
It remains to show that the link (i.e. cross-section) of each of these tangent cones is connected or, in other words, that the cone minus the tip is connected.

Fix some $i \in \{ 1, \ldots, N \}$ and assume that the link of the tangent cone around $y_{i, \infty}$ was not connected.
Choose points $z', z'' \in X$ close to $y_{i, \infty}$, but in regions corresponding to different components in the link, such that any minimizing geodesic between $z'$ and $z''$ has to pass through $y_{i, \infty}$.
Let $z'_k, z''_k \in (\M, \lambda_k^2 g_{t_k}, x_k)$ be sequences such that $z'_k \to z'$ and $z''_k \to z''$ and let $\gamma_k \subset \M$ be time-$t_k$ minimizing geodesics between $z'_k$ and $z''_k$.
Then, after passing to a subsequence, the $\gamma_k$ converge to a minimizing geodesic $\gamma_\infty \subset X$ passing through $y_{i, \infty}$.
It follows that $\min_{\gamma_k} \rrm (\cdot, t_k) \to 0$ (where $\rrm ( \cdot, t_k)$ is determined with respect to $\lambda_k^2 g_{t_k}$), because otherwise $X$ would be smooth around $y_{i, \infty}$.
Let $x'_k \in \gamma_k$ such that $\rrm (x'_k, t_k) = \min_{\gamma_k} \rrm (\cdot, t_k)$ and set $\lambda'_k := \rrm^{-1} (x'_k, t_k)$.
Then, after passing to a subsequence, $(\M, \lambda^{\prime 2}_k g_{t_k}, x'_k)$ converges to a blowup limit $(X', d', x'_\infty)$ containing a geodesic line $\gamma'_\infty \subset X'$, which passes through $x'_\infty$.
By Lemma \ref{Lem:Xsmoothonreg}, $(X', d')$ is described by a Riemannian metric $g'_\infty$ away from finitely many singular points $y'_{1, \infty}, \ldots, y'_{N', \infty} \in X'$, which is Ricci flat.
By the choice of $\lambda'_k$, we know that $|{\Rm}_{g'_\infty} | \leq 1$ in a $1$-neighborhood around $\gamma'_\infty$ and that $\gamma'_\infty$ does not hit any singular points.
We also get that $\rrm (x'_\infty) = 1$ if $(X', d')$ is smooth.

Denote by $(\overline{X}', \overline{d}')$ the completion of length metric induced by the Riemannian metric $g'_\infty$ on $X' \setminus \{ y'_{1, \infty}, \ldots, y'_{N', \infty} \}$.
So $(\overline{X}', \overline{d}')$ arises from $(X', d')$ by ``splitting the tangent cones around singular points into tangent cones with connected links''.
So the tangent cones of $(\overline{X}', \overline{d}')$ are finite quotients of $\IR^4$.
It follows that any minimizing geodesic in $(\overline{X}', \overline{d}')$ can only hit singular points at its endpoints.
Recall that $\Ric_{g'_\infty} \equiv 0$ and that $(\overline{X}', \overline{d}')$ contains a line.
The proof of the Cheeger-Gromoll splitting theorem (cf \cite{CG}) still works in this setting and we obtain a non-constant smooth function $f \in C^\infty (X' \setminus \{ y'_{1, \infty}, \ldots, y'_{N', \infty} \})$ with parallel gradient.
The existence of such a function implies that $(\overline{X}', \overline{d}')$ has no singular points and hence $(X', d') \cong (\overline{X}', \overline{d}')$ isometrically splits off a line.
So $(X', d')$ is flat and has no singular points, contradicting our earlier conclusions.
This finishes the proof.
\end{proof}

We can finally prove Corollary \ref{Cor:convergenceorbifold}.

\begin{proof}[Proof of Corollary \ref{Cor:convergenceorbifold}]
We define the function $\rrm^\infty : \M \to [0, \infty)$ as follows:
\[ \rrm^\infty (x) = \limsup_{t \nearrow T} \rrm (x,t). \]
Consider a point $x \in \M$ with $r := \rrm^\infty (x) > 0$.
So there is a sequence of times $t_k \nearrow T$ such that $\rrm (x, t_k) > r /2$.
Using the backward pseudolocality Theorem \ref{Thm:backwpseudoloc}, we find that
\[ |{\Rm}| < 4 Kr^{-2} \qquad \text{on} \qquad  \bigcup_{k=1}^\infty P(x,t_k, \tfrac12 \eps r, - (\tfrac12 \eps r)^2). \]
By a distance distortion estimate, there is a small $\tau > 0$, depending on $r$, such that
\[ P (x,T-\tau, \tfrac18 \eps r, \tau) \subset \bigcup_{t \in [T-\tau, T)} B(x,t, \tfrac14 \eps r) \times \{ t \} \subset  \bigcup_{k=1}^\infty P(x,t_k, \tfrac12 \eps r, - (\tfrac12 \eps r)^2). \]
So as $t \nearrow T$, the metric $g_t$ converges smoothly to a Riemannian metric $g_T$ in a neighborhood of $x$.
Moreover, $\rrm^\infty$ is positive in a neighborhood of $x$.
We hence obtain that $\M^{\textnormal{reg}} := \{ \rrm^\infty > 0 \} \subset \M$ is an open subset and that $g_t$ converges smoothly to a Riemannian metric $g_T$ on $\M^{\textnormal{reg}}$ as $t \nearrow T$.
Moreover, $\rrm^\infty$ restricted to $\M^{\textnormal{reg}}$ is $1$-Lipschitz with respect to $g_T$, being the limsup of a family of $1$-Lipschitz functions.
So by continuity of the metric $g_t$, for any $x \in \M^{\textnormal{reg}}$ and $\delta > 0$ the ball $B(x,t,\rrm^\infty(x) - \delta)$ is contained in $\M^{\textnormal{reg}}$ for $t$ sufficiently close to $T$.
We also have $|{\Rm}| (\cdot, T) \leq (\rrm^\infty (x))^{-2}$ on $B(x,t,\rrm^\infty(x) - \delta)$ for each $\delta > 0$.
So 
\[ \rrm^\infty (x) = \lim_{t \nearrow T} \rrm (x,t) \qquad \text{for all} \qquad x \in \M. \]

Next, we show that $\M^{\textnormal{sing}} := \M \setminus \M^{\textnormal{reg}} = \{ \rrm^\infty = 0 \}$ is a null set with respect to $g_t$ for any $t \in [0,T)$.
Recall that $C_1^{-1} dg_{t_1} < dg_{t_2} < C_1 dg_{t_1}$ for any $t_1, t_2 \in [0, T)$.
So it suffices to show that $\M^{\textnormal{sing}}$ is a null set with respect to $dg_0$.
Since $\lim_{t \nearrow T} \rrm (\cdot, t) = 0$ on $\M^{\textnormal{sing}}$, we have with (\ref{eq:rrmvolboundC2}) that for any $0 < s < 1$
\begin{multline*}
 |\M^{\textnormal{sing}} |_0 \leq |\{ \rrm^\infty < s \} |_0 =  |\{ \limsup_{t \nearrow T} \rrm (\cdot, t) < s \}|_0 = \limsup_{t \nearrow T} |\{\rrm (\cdot, t) < s \}|_0 \\
 \leq C_1 \limsup_{t \nearrow T} |\{\rrm (\cdot, t) < s \}|_t \leq C_1 C_5 s^4.
\end{multline*}
Letting $s \to 0$ yields $|\M^{\textnormal{sing}}|_0 = 0$, which establishes assertion (b).

Consider now the points $y_{1,t}, \ldots, y_{N, t}$ from Lemma \ref{Lem:yit}.
We claim that for any $0 < s < 1$ and any $t < T$ sufficiently close to $T$, we have
\begin{equation} \label{eq:Msingballs}
 \M^{\textnormal{sing}} \subset \bigcup_{i=1}^N B(y_{i,t}, t, s).
\end{equation}
Fix $s$ and choose $s' > 0$ such that $s' < \frac1{40} s$ and $| \{ \rrm^\infty < s' \} |_t < \kappa_1 (\tfrac1{2} s)^4$ for any $t \in [0,T)$ (this is possible because $\M^{\textnormal{sing}}$ is a nullset and the volume form $dg_t$ has bounded distortion).
Define $U := \{ \rrm^\infty \geq s' \} \subset \M$ and choose $\tau > 0$ small enough such that $\rrm > \frac12 s'$ on $U \times [T-\tau, T)$ and $\rrm < 2s'$ on $\partial U \times [T-\tau, T)$.
Let now $t \in [T-\tau, T)$ and $x \in \M^{\textnormal{sing}} \subset \M \setminus U$.
We claim that there is a point $z \in \partial U$ such that $d_t (x,z) < \frac1{2} s$.
Otherwise $B(x, t, \tfrac1{2} s) \cap U = \emptyset$, which would imply that
\[ \kappa_1 (\tfrac1{2} s)^4 < | B(x,t,s) |_t \leq | \M \setminus U |_t < \kappa_1 (\tfrac1{2} s)^4, \]
which is a contradiction.
Since $\rrm (z,t) < 2 s' < \tfrac1{20} s$, there is an $i \in \{ 1, \ldots, N \}$ such that $d_t (z, y_{i,t}) <  6 \cdot \frac{1}{20} s < \frac1{2} s$, by Lemma \ref{Lem:yit}.
It follows that $x \in B(y_{i,t}, t, s)$.
This proves (\ref{eq:Msingballs}).

Inclusion (\ref{eq:Msingballs}) implies that $(\M, g_t)$ converges to a metric space $(X, d)$ in the Gromov-Hausdorff sense, which is isometric to $(\M^{\textnormal{reg}}, g_\infty)$ away from at most $N$ points.
By Lemma \ref{Lem:itsanorbifold}, the tangent cones of $(X,d)$ are finite quotients of $\IR^4$.
It follows that $\M^{\textnormal{reg}}$ is connected and that $(X,d)$ is isometric to the metric completion $(\overline{\M}^{\textnormal{reg}}, d^{g_T})$ of $(\M^{\textnormal{reg}}, g_T)$.
Hence $\overline{\M}^{\textnormal{reg}}$ can be endowed with a smooth orbifold structure such that $g_T$ is a smooth metric away from the singular points.
This shows assertion (a)--(d).

Let us now show assertion (e).
Fix a singular point $p \in \overline{\M}^{\textnormal{reg}}$ and denote by $r(x) := d^{g_T} (p, x)$ the radial distance.
Since blow-ups around $p$ smoothly converge to a finite quotient of $\IR^4$, we find that $|\nabla^m {\Rm}| < o(r^{-2-m})$.
So $\rrm^\infty (x) > c r(x)$ close to $p$ for some uniform $c > 0$.
So Lemma \ref{Lem:goodRicbound} implies that $|\nabla^m {\Ric}| < C( \rrm^\infty)^{-1-m} < O(r^{-1-m})$ for all $m \geq 0$.
We now construct the orbifold metric $\overline{g}_\eps$ for some given $\eps$, following the lines of \cite{BKN}.
Denote the tangent cone of $(\overline{\M}^{\textnormal{reg}}, d^{g_T})$ at $p$ by $(\M^{\textnormal{tang}} = \IR^4 / \Gamma, g_{\textnormal{eucl}})$, $\Gamma \subset O(4)$, let $p' \in \M^{\textnormal{tang}}$ be its vertex and $r' : \M^{\textnormal{tang}} \to [0, \infty)$ the Euclidean distance from its tip.
Define the annuli
\[ A' (r_1, r_2) := \{ x \in \M^{\textnormal{tan}} \;\; : \;\; r_1 < r'(x) < r_2 \} \subset \M^{\textnormal{tan}} \]
and
\[ A (p, r_1, r_2) := \{ x \in \M^{\textnormal{reg}} \;\; : \;\; r_1 < d^{g_T} (p, x) < r_2 \} \subset \M^{\textnormal{reg}}  \]
and set for $i = 0, 1, \ldots$
\[ A_i := A(p, 2^{-i-1}, 2^{-i+1}). \]
Since the tangent cone of $(\overline{\M}^{\textnormal{reg}}, d^{g_T})$ at $p$ is isometric to $(\M^{\textnormal{tang}}, g_{\textnormal{eucl}})$ we can find a sequence of Riemannian metrics $\overline{g}^i$ on each $A_i$ such that the following holds for large $i$:
There is an isometry $\Phi_i : A_i \to U_i \subset \M^{\textnormal{tang}}$ with $\overline{g}^i = \Phi_i^* g_{\textnormal{eucl}}$ such that
\[ A' (2^{-i-0.9}, 2^{-i+0.9} ) \subset U_i \subset A' (2^{-i-1.1}, 2^{-i+1.1}) \]
and such that
\[ \lim_{i \to \infty} \Vert g_T - \overline{g}^i \Vert_{C^0 (A_i; \overline{g}^i)} = 0. \]
More specifically, since $|\nabla^m {\Ric}| < O(r^{-1-m})$, we can use a harmonic coordinate construction similar to the one carried out in \cite{BKN}, and choose the metrics $\overline{g}^i$ such that
\begin{align*}
 \Vert g_T - \overline{g}^i & \Vert_{C^0 (A_i; \overline{g}^i)} + 2^{-i} \Vert  \partial (g_T - \overline{g}^i) \Vert_{C^0 (A_i; \overline{g}^i)}  + (2^{-i} )^2 \Vert \partial^2(g_T - \overline{g}^i) \Vert_{C^0 (A_i; \overline{g}^i)}  \\
& \leq C (2^{-i})^2 \Vert {\Rm} \Vert_{L^\infty (A_{i-1} \cup A_i \cup A_{i+1}; g_T)} 
 + C (2^{-i})^2 \Vert  {\Ric} \Vert_{L^\infty (A_{i-1} \cup A_i \cup A_{i+1}; g_T)} \\
 &\qquad +  C (2^{-i})^3 \Vert  \nabla {\Ric} \Vert_{L^\infty (A_{i-1} \cup A_i \cup A_{i+1}; g_T)} \\
 &\leq  C (2^{-i})^2 \Vert {\Rm} \Vert_{L^\infty (A_{i-1} \cup A_i \cup A_{i+1}; g_T)}  + C 2^{-i} =: \delta_i \to 0, 
\end{align*}
for some generic constant $C < \infty$.
So the change of coordinate maps restricted to the annuli $A'(2^{-i-0.9}, 2^{-i-0.1})$
\[  \Phi_{i+1}  \circ \Phi_{i}^{-1} : A'(2^{-i-0.9}, 2^{-i-0.1})  \to A'(2^{-i-2.1}, 2^{-i+0.1}) \]
can be approximated by isometries $\psi_i : \M^{\textnormal{tang}} \to \M^{\textnormal{tang}}$, $\psi_i^* g_{\textnormal{eucl}} = g_{\textnormal{eucl}}$ in such a way that
\begin{equation} \label{eq:psiPhiPhi}
 \Vert \psi_{i+1}^{-1} \circ \Phi_{i+1}  \circ \Phi_{i}^{-1} \Vert_{C^3(A'(2^{-i-0.8}, 2^{-i-0.2}))} \leq C \delta_i 
\end{equation}
for some uniform constant $C < \infty$.
Replace now $\Phi_i$ by $\psi_1^{-1} \circ \ldots \circ \psi_i^{-1} \circ \Phi_i$.
Then still $\overline{g}^i = \Phi_i^* g_{\textnormal{eucl}}$ and (\ref{eq:psiPhiPhi}) simplifies to 
\[ \Vert \Phi_{i+1}  \circ \Phi_{i}^{-1} \Vert_{C^3(A'(2^{-i-0.8}, 2^{-i-0.2}))} \leq C \delta_i \]
So it is possible to glue these diffeomorphisms together to a diffeomorphism $\Phi : U \setminus \{ p \} \to V \setminus \{ p' \}$, where $U \subset \overline{\M}^{\textnormal{reg}}$ is an open neighborhood of $p$ and $V \subset \M^{\textnormal{tan}}$ is an open neighborhood of $p'$, such that the pullback $\overline{g} := \Phi^* g_{\textnormal{eucl}}$ satisfies
\begin{multline} \label{eq:ovgannuli}
  \Vert g_T - \overline{g} \Vert_{C^0 (A_i; \overline{g})} + 2^{-i} \Vert  \partial (g_T - \overline{g}) \Vert_{C^0 (A_i; \overline{g})}  + (2^{-i} )^2 \Vert \partial^2(g_T - \overline{g}) \Vert_{C^0 (A_i; \overline{g})} \\ \leq C (\delta_{i-1} + \delta_i + \delta_{i+1} ). 
\end{multline}

We now show that $\Vert g_T - \overline{g} \Vert_{W^{2,2} (U; \overline{g})} < \infty$.
By (\ref{eq:ovgannuli}) it suffices to show that $\sum_{i=1}^\infty \delta_i^2 < \infty$.
To see this, note that by the bound $|\nabla^m {\Ric}| < O(r^{-1-m})$ and Anderson's $\eps$-regularity theorem (cf \cite{Anderson-epsreg}) we have for large $i$
\begin{align*}
 \Vert {\Rm} \Vert_{L^\infty (A_{i-1} \cup A_i \cup A_{i+1} ; g_T )} &\leq C(2^{-i})^{-2} \Vert {\Rm} \Vert_{L^2 (A_{i-2} \cup \ldots \cup A_{i+2}; g_T)} \\
 &\qquad\qquad\qquad + C \Vert {\Ric} \Vert_{L^\infty (A_{i-2} \cup \ldots \cup A_{i+2}; g_T)} \\
 &\leq C(2^{-i})^{-2} \Vert {\Rm} \Vert_{L^2 (A_{i-2} \cup \ldots \cup A_{i+2}; g_T)} + C (2^{-i})^{-1}.
\end{align*}
So
\[ \delta_i^2 \leq C \Vert {\Rm} \Vert_{L^2 (A_{i-2} \cup \ldots \cup A_{i+2}; g_T)}^2 + C (2^{-i})^2, \]
which implies $\sum_{i=1}^\infty \delta_i^2 < \infty$.

We finally construct the metric $\overline{g}_\eps$.
For this let $i_0 < \infty$ be some large integer whose value we will determine later, depending on $\eps$ and set $\rho := 2^{-i_0}$.
Choose a smooth cutoff function $\eta : \M^{\textnormal{reg}} \to [0,1]$ such that $\eta \equiv 1$ on $B(p, \rho)$ and $\eta \equiv 0$ outside $B(p, 2 \rho)$ and such that $|\nabla \eta| < C \rho^{-1}$ and $|\nabla^2 \eta| < C \rho^{-2}$ for some $C < \infty$ that does not depend on $\rho$.
Then set 
\[ \overline{g}_\eps := \eta \overline{g} + ( 1- \eta  ) g_T. \]
Using (\ref{eq:ovgannuli}) we one can see that
\[ \Vert g_T - \overline{g}_\eps \Vert_{C^0(A_i; \overline{g}_\eps)} \leq C \sup \{ \delta_{i_{0}-1}, \delta_{i_0 }, \ldots \} \] 
and
\begin{multline*}
 \Vert g_T - \overline{g}_\eps \Vert_{W^{2,2}(A_i; \overline{g}_\eps)} \leq C  \Vert g_T - \overline{g} \Vert_{W^{2,2}(A_{i_0} \cup A_{i_0 + 1} \cup \ldots ; \overline{g})} + C \rho^{-1} \Vert \partial (g_T - \overline{g}) \Vert_{L^2(A_{i_0}  ; \overline{g})} \\
 + C \rho^{-2} \Vert g_T - \overline{g} \Vert_{L^2(A_{i_0}  ; \overline{g})}  \\
\leq C  \Vert g_T - \overline{g} \Vert_{W^{2,2}(A_{i_0} \cup A_{i_0 + 1} \cup \ldots ; \overline{g})} + C \rho^{-1} \Vert \partial (g_T - \overline{g}) \Vert_{L^2(A_{i_0}  ; \overline{g})} + C ( \delta_{i_0-1}+ \delta_{i_0}+ \delta_{i_0+1}).
\end{multline*}
By choosing $i_0$ sufficiently large, we can ensure that the the right-hand sides of both inequalities smaller than $\eps/2$.
This proves assertion (e) and hence the corollary.
\end{proof}

\section*{Acknowledgements}
The first author would like to thank Simon Brendle, Panagiotis Gianniotis, John Lott, Yong Wei.
The second author would like to thank Professors Xiuxiong Chen, Gang Tian, Bing Wang, Zhenlei Zhang and Xiaohua Zhu for helpful conversations.


\begin{thebibliography}{00000}
\bibitem[And1]{Anderson-epsreg} Anderson, Michael T., {\it Ricci curvature bounds and Einstein metrics on compact manifolds}, J. Amer. Math. Soc. 2 (1989), no. 3, 455-490. 

\bibitem[And2]{Anderson} Anderson, Michael T., {\it A survey of Einstein metrics on 4-manifolds}, Handbook of geometric analysis, No. 3, 1-39, Adv. Lect. Math. (ALM), 14, Int. Press, Somerville, MA, 2010, http://arxiv.org/abs/0810.4830. 

\bibitem [BCG]{BCG:1} Barlow, Martin; Coulhon, Thierry; Grigor'yan, Alexander,
{\it Manifolds and graphs with slow heat kernel decay}, Invent. Math. 144 (2001), no. 3, 609-649.

\bibitem[CG]{CG} Cheeger, Jeff; Gromoll, Detlef,
{\it The splitting theorem for manifolds of nonnegative Ricci curvature}, 
J. Differential Geometry 6 (1971/72), 119-128. 

\bibitem [CH]{CH:1} Cao, Xiaodong; Hamilton, Richard S., {\it Differential Harnack estimates for time-dependent heat equations with potentials},  Geom. Funct. Anal. 19 (2009), no. 4, 989-1000.

\bibitem [ChH]{ChH:1} Chow, Bennett; Hamilton, Richard S., {\it Constrained and linear Harnack inequalities for parabolic equations}, Invent. Math. 129 (1997), no. 2, 213-238.

\bibitem[CN]{Cheeger-Naber-codim-4} Cheeger, Jeff; Naber, Aaron, {\it Regularity of Einstein Manifolds and the Codimension 4 Conjecture}, http://arxiv.org/abs/1406.6534 (2014).

\bibitem[CT]{CH:2} Cao, Xiaodong; Tran, Hung, {\it Mean Value Inequalities and Conditions to Extend Ricci Flow}, http://arxiv.org/abs/1303.4492v1 (2013).

\bibitem [CW1]{CW:1} 
Chen, Xiuxiong; Wang, Bing, {\it Space of Ricci flows I},
Comm. Pure Appl. Math. 65 (2012), no. 10, 1399-1457, http://arxiv.org/abs/0902.1545

\bibitem [CW2]{CW:2} Chen, Xiuxiong; Wang, Bing, {\it On the conditions to extend Ricci flow(III)},
Int. Math. Res. Not. IMRN 2013, no. 10, 2349-2367, http://arxiv.org/abs/1107.5110

\bibitem [CW3]{CW:3} Chen, Xiuxiong; Wang, Bing, {\it Space of Ricci flows (II)}, http://arxiv.org/abs/1405.6797

\bibitem[GH]{GH:2} Grigor'yan, Alexander; Hu, Jiaxin, {\it Off-diagonal upper estimates for the heat kernel of the Dirichlet forms on metric spaces}, Invent. Math. 174 (2008), no. 1, 81-126.

\bibitem[Gr]{Gr:2} Grigor'yan, Alexander, {\it Heat kernel and analysis on manifolds}, AMS/IP Studies in Advanced Mathematics, 47. American Mathematical Society, Providence, RI; International Press, Boston, MA, 2009. xviii+482 pp.

\bibitem [Ha1]{Ha:0}Hamilton, Richard S., {\it Three-manifolds with positive Ricci curvature},
 J. Differential Geom. 17 (1982), no. 2, 255-306.

\bibitem [Ha2]{Ha:1} Hamilton, Richard S., {\it The formation of singularities in the Ricci flow}, Surveys in differential geometry, Vol. II (Cambridge, MA, 1993), 7-136, Int. Press, Cambridge, MA, 1995.

\bibitem[HN]{HN:1} Hein, Hans-Joachim; Naber, Aaron, {\it New logarithmic Sobolev inequalities and an $\varepsilon$-regularity theorem for the Ricci flow}, Comm. Pure Appl. Math. 67 (2014), no. 9, 1543-1561.

\bibitem[BKN]{BKN} Bando, Shigetoshi; Kasue, Atsushi; Nakajima, Hiraku, {\it On a construction of coordinates at infinity on manifolds with fast curvature decay and maximal volume growth},
Invent. Math. 97 (1989), no. 2, 313-349. 

\bibitem[L]{L:1} Li, Peter, {\it Geometric analysis}, Cambridge Studies in Advanced Mathematics, 134. Cambridge University Press, Cambridge, 2012. 

\bibitem[LT]{LT:1} Li, Peter; Tian, Gang, {\it On the heat kernel of the Bergmann metric on algebraic varieties}, J. Amer. Math. Soc., 8 (1995), 857-877.

\bibitem [LY]{LY:1} Li, Peter; Yau, Shing-Tung, {\it On the parabolic kernel of the Schr\"odinger operator.}  Acta Math. 156 (1986), no. 3-4, 153-201.

\bibitem[P1]{P:1} Perelman, Grisha, {\it The entropy formula for the Ricci flow and its
geometric applications}, http://arxiv.org/abs/math/0211159 (2002).

\bibitem[P2]{P:2} Perelman, Grisha, {\it Ricci flow with surgery on three-manifolds}, http://arxiv.org/abs/math/0303109 (2003).

\bibitem[Se]{Sesum-Ricci}Sesum, Natasa, {\it Curvature tensor under the Ricci flow}, Amer. J. Math. 127 (2005), no. 6, 1315-1324.

\bibitem[Sh]{Sh:1} Shi, Wan-Xiong, {\it Deforming the metric on complete Riemannian manifolds}, J. Differential Geom. 30 (1989), no. 1, 223-301. 

\bibitem[Si1]{Si:1} Simon, Miles, {\it  Ricci flow of almost non-negatively curved three manifolds}, J. Reine Angew. Math. 630 (2009), 177-217.

\bibitem[Si2]{Simon:1} Simon, Miles, {\it Some integral curvature estimates for the Ricci flow in four dimensions}, http://arxiv.org/abs/1504.02623 (2015).

\bibitem[Si3]{Simon:2} Simon, Miles, {\it Extending four dimensional Ricci flows with bounded scalar curvature}, http://arxiv.org/abs/1504.02910 (2015).

\bibitem[ST]{ST:1} Sesum, Natasa; Tian, Gang, {\it Bounding scalar curvature and diameter along the K\"ahler Ricci flow (after Perelman)}, J. Inst. Math. Jussieu 7 (2008), no. 3, 575-587.

\bibitem[TW]{Tian-Wang} Tian, Gang; Wang, Bing,
{\it On the structure of almost Einstein manifolds}, http://arxiv.org/abs/1202.2912 (2012).

\bibitem[TZz]{TZz:1} Tian, Gang; Zhang, Zhenlei, 
{\it Regularity of  K\"ahler Ricci flows on Fano manifolds}, http://arxiv.org/abs/1310.5897v1 (2013).

\bibitem[Wa]{Wa:1} Wang, Bing, {\it On the conditions to extend Ricci flow (II)},
 Int. Math. Res. Not. IMRN 2012, no. 14, 3192-3223.

\bibitem[Ye]{Ye:1} Rugang Ye, {\it The logarithmic Sobolev inequality along the Ricci flow}, http://arxiv.org/abs/0707.2424 (2007).

\bibitem[Z06]{Z06:1} Zhang, Qi S., {\it Some gradient estimates for the heat equation on domains and for an equation by Perelman}, Int. Math. Res. Not., 39 pages Art. ID 92314, 39, 2006.

\bibitem[Z07]{Z07:1} Zhang, Qi S., {\it A uniform Sobolev inequality under Ricci flow},
IMRN 2007, ibidi Erratum, Addendum.

\bibitem[Z11]{Z11:1} Zhang, Qi S., {\it Bounds on volume growth of geodesic balls under Ricci flow}, Math. Res. Lett. 19 (2012), no. 1, 245-253, http://arxiv.org/abs/1107.4262

\bibitem[ZZh]{ZhouZhang-scal} Zhang, Zhou, {\it Scalar curvature behavior for finite-time singularity of K\"ahler-Ricci flow}, Michigan Math. J. 59 (2010), no. 2, 419-433.
\end{thebibliography}
\end{document}